\documentclass[12pt,reqno]{amsart}

\usepackage[OT1]{fontenc}
\RequirePackage{times}

\usepackage{amssymb,amsmath,amsfonts,eurosym}
\usepackage[mathscr]{euscript}
\usepackage{enumitem,xcolor}
\allowdisplaybreaks
\def\D {{\mathfrak D}}

\def\H {{\mathcal H}}
\def\V {{\mathcal V}}
\def\Z {{\mathcal Z}}

\def\M {{\mathcal M}}
\def\N {{\mathcal N}}

\def\R {\mathbb{R}}
\def\Re {\mathfrak{Re\,}}

\def\eps{\varepsilon}
\def\e{{\rm e}}
\def\d{{\rm d}}
\def\ddt{\frac{\d}{\d t}}
\def\i{{\rm i}}
\def \l {\langle}
\def \r {\rangle}
\def \and {{\qquad\text{and}\qquad}}

\newtheorem{proposition}{Proposition}[section]
\newtheorem{theorem}[proposition]{Theorem}
\newtheorem{corollary}[proposition]{Corollary}
\newtheorem{lemma}[proposition]{Lemma}
\theoremstyle{definition}

\newtheorem{remark}[proposition]{Remark}
\numberwithin{equation}{section}
\def \au {\rm}
\def \ti {\it}
\def \jou {\rm}
\def \bk {\it}
\def \no#1#2#3 {{\bf #1} (#3), #2.}
\def \eds#1#2#3 {#1, #2, #3.}

\title[Bresse and Timoshenko]
{On the stability of Bresse and Timoshenko systems\\ with hyperbolic heat conduction}

\author[F. Dell'Oro]
{Filippo Dell'Oro}

\address{Politecnico di Milano - Dipartimento di Matematica
\newline\indent
Via Bonardi 9, 20133 Milano, Italy}
\email{filippo.delloro@polimi.it}


\begin{document}

\begin{abstract}
We investigate the stability of three thermoelastic beam systems with hyperbolic heat conduction.
First, we study the Bresse-Gurtin-Pipkin system,
providing a necessary and sufficient condition for the exponential stability
and the optimal polynomial decay rate
when the condition is violated. Second, we obtain
analogous results for the Bresse-Maxwell-Cattaneo system, completing
an analysis recently initiated in the literature.
Finally, we consider the Timoshenko-Gurtin-Pipkin system and
we find the optimal polynomial decay rate when the known
exponential stability condition does not hold.
As a byproduct, we fully recover the stability characterization of the
Timoshenko-Maxwell-Cattaneo system.
The classical "equal wave speeds" conditions
are also recovered through singular limit procedures.
Our conditions are
compatible with some physical constraints on the coefficients as the positivity of the Poisson's ratio of the material.
The analysis faces several challenges connected with the
thermal damping, whose
resolution rests on recently
developed mathematical tools such as quantitative Riemann-Lebesgue lemmas.
\end{abstract}

\maketitle

\section{Introduction}

\subsection{Preamble}\label{PRE}
In 1993-1994, J.E.~Lagnese, G. Leugering and E.J.P.G. Schmidt derived a general
nonlinear PDE model for
the dynamics of thin thermoelastic beams~\cite{LLS,LLS2}. A particular linearized
case of such a model is the Bresse-Fourier (BF) system
\begin{equation}
\label{bresse-fou}
\begin{cases}
\rho_1 \varphi_{tt} -k(\varphi_x +\psi +l w)_x - l k_0(w_x - l\varphi) + l\gamma \xi= 0,\\\noalign{\vskip0.3mm}
\rho_2 \psi_{tt} -b\psi_{xx} +k(\varphi_x +\psi +lw) +\gamma\vartheta_x= 0,\\\noalign{\vskip0.3mm}
\rho_1 w_{tt}  -k_0(w_x - l\varphi)_x + lk(\varphi_x +\psi +lw) +\gamma \xi_x= 0,\\\noalign{\vskip0.3mm}
\rho_3\vartheta_t -\varpi \vartheta_{xx}+ \gamma \psi_{xt}=0,\\\noalign{\vskip0.3mm}
\rho_3\xi_t - \varpi\xi_{xx} + \gamma (w_{xt} - l \varphi_t)=0,\
\end{cases}
\end{equation}
which describes the vibrations of a curved thin thermoelastic beam of length $\ell>0$, taking into account
both rotatory inertia and shear deformation effects.
The unknowns $\varphi,\psi,w$ represent
the vertical displacement, the rotation angle of the cross-section and
the horizontal displacement, while $\vartheta,\xi$
represent the temperature (deviations from a fixed reference temperature) along the vertical and horizontal directions.
With standard notation,
the subscripts $t$ and $x$ indicate the partial derivatives with respect
to the time variable $t>0$ and the space variable $x\in (0,\ell)$.
The strictly positive constants
$\rho_1,\rho_2, \rho_3, k, k_0, b, \varpi, l,\gamma$ account for the physical properties of the beam (see
\cite{LLS,LLS2,LIU} for details).
In particular, they are subjected to the constraints
\begin{equation}
\label{phydef}
k_0=\frac{b\rho_1}{\rho_2} \,\,\,\and\,\,\,
b>\frac{k \rho_2}{\rho_1}.
\end{equation}
The first equality tells that the rotation angle and the horizontal
displacement motions have the same wave speeds. The second relation tells that the wave speed of
the rotation angle equation is greater than the one of the vertical displacement equation, and
is a consequence of the positivity of
the Poisson's ratio of the involved material (which is always the
case for ``ordinary" media). Finally, the constant $l$
accounts for the initial curvature of the beam.
However, as customary in the mathematical literature on the subject, in the sequel
we will unhook these constants from their physical meaning,
allowing them to assume any positive real value.
Still, we will keep in mind \eqref{phydef} as it permits
a physical interpretation of the mathematical results.

\begin{remark}
System \eqref{bresse-fou} takes the first part of the name (Bresse) from the fact
that its isothermal counterpart (i.e.\ when the temperatures are neglected)
was derived in 1859 by the French engineer
J.A.C. Bresse in his pioneering work \cite{BRR}. The second part of the name (Fourier) is because
the temperature evolution is modeled using the Fourier law (see~\cite{LLS,LLS2}
for details).
\end{remark}

It is interesting to observe that in the limit case when $l=0$, namely when the
beam is straight, the horizontal
displacement uncouples from the vertical and the rotation angle motions.
In this situation \eqref{bresse-fou} splits into the Timoshenko-Fourier (TF) system
\begin{equation}
\label{TF}
\begin{cases}
\rho_1\varphi_{tt} -k (\varphi_{x}+\psi)_x=0,\\\noalign{\vskip0.3mm}
\rho_2 \psi_{tt}- b \psi_{xx} + k(\varphi_x+\psi)+\gamma \vartheta_x=0,\\\noalign{\vskip0.3mm}
\rho_3\vartheta_t -\varpi \vartheta_{xx}+ \gamma \psi_{xt}=0,
\end{cases}
\end{equation}
and the system of second-order linear thermoelasticity in one dimension
\begin{equation}
\label{IITHERMO}
\begin{cases}
\rho_1 w_{tt}  -k_0w_{xx} +\gamma \xi_x= 0,\\\noalign{\vskip0.3mm}
\rho_3\xi_t - \varpi\xi_{xx} + \gamma w_{xt}=0.
\end{cases}
\end{equation}
Therefore, one may assert that \eqref{TF} and \eqref{IITHERMO}
are both special cases of~\eqref{bresse-fou}. Of course, since the
constant $l$ is assumed to be strictly positive, one should more properly say that
\eqref{TF} and \eqref{IITHERMO} can be obtained from \eqref{bresse-fou} in the (singular) limit $l\to0$.

\begin{remark}
System \eqref{TF} takes the first part of the name from the fact
that its isothermal counterpart is
the well-known Timoshenko system~\cite{TIMOS}
or, more properly, the Timoshenko-Ehrenfest system (see the historical account \cite{ELI}).
\end{remark}

\subsection{Stability of the BF and TF systems}
\label{stabbf}

The stability properties of the BF system \eqref{bresse-fou}
have been analyzed for the first time by Z. Liu and B. Rao in the influential paper \cite{LIU}. There, the
authors introduced two stability numbers
$$
\chi_0 = b - \frac{k\rho_2}{\rho_1} \,\and\, \chi_1= k_0-k,
$$
and showed that the solution semigroup associated to \eqref{bresse-fou} with appropriate boundary conditions
is exponentially stable if and only if
\begin{equation}
\label{nscBF}
\chi_0 \hspace{0.4mm} \chi_1=0.
\end{equation}
This means that exponential stability occurs if and only if the first and the second equation of~\eqref{bresse-fou} have
the same wave speed or the same do the first and the third equation. When \eqref{nscBF}
is violated, it was proved in
\cite{LIU} that the semigroup is polynomially stable. Note that,
since \eqref{phydef} is incompatible with \eqref{nscBF}, exponential stability never occurs
in physical situations.

Concerning the TF system \eqref{TF}, it was proved by J.E. Mu{\~n}oz Rivera
and R. Racke \cite{RRTIM} that the associated solution semigroup is exponentially stable if and only if
$\chi_0=0$. Again, the latter condition
is physically unrealistic \cite{OLSON} and, when violated, leads to a polynomially stable
dynamics \cite{CARDO}.

\begin{remark}
It is a classical result that the semigroup associated
to the system of second-order thermoelasticity \eqref{IITHERMO} is exponentially stable
independently of the value of the
structural constants in the model (see e.g.\ \cite[Chapter 2]{LZ} and \cite[Chapter 2]{RACK}).
\end{remark}

The meaning of \eqref{nscBF} has been highlighted in \cite{LIU},
where it is explained that
the variables $\psi$ and $w$ are ``effectively damped"
due to the coupling with the temperatures but the variable $\varphi$ is only
``indirectly damped" via the second and third equations plus a weaker coupling with the temperature.
Condition \eqref{nscBF} tells that the
effectiveness of the damping acting on $\varphi$ depends on the equality between the wave speeds
of either the second (effectively damped) equation and the first equation, or
the third (effectively damped) equation and the first equation. From the technical side,
condition~\eqref{nscBF} provides a cancellation of some higher-order terms that cannot
be controlled with the first-order energy.
A similar phenomenon appears in
the TF system.

\subsection{Hyperbolic heat conduction}\label{hypsec}
In the physical derivation of system \eqref{bresse-fou} one employs
the classical Fourier thermal law
\begin{equation}
\label{foulaw}
p=-\varpi \vartheta_x,
\end{equation}
where $p=p(x,t)$ is the so-called heat-flux variable (see \cite{LLS,LLS2} for details).
But, as is well-known, the use of \eqref{foulaw} leads to an infinite speed
propagation of thermal signals, due to the parabolic character of the heat equation.
On the contrary, with the advent
of modern microscale technologies, there is an increasing evidence that the thermal motion
is a wave-type phenomenon, where
the temperature may travel with a finite speed of propagation
(see e.g.\ \cite{St} and references therein). As a consequence,
a number of ``hyperbolic heat conduction theories"
have been proposed along the years.
One of them is due to M.E. Gurtin and A.C. Pipkin \cite{GP} and consists in replacing \eqref{foulaw} with
\begin{equation}
\label{GPLAW}
p(t) = - \varpi\int_0^\infty g(s)\vartheta_{x}(t-s)\d s,
\end{equation}
where $g$ is a suitable convolution kernel. Equation \eqref{GPLAW} is a memory
relaxation of~\eqref{foulaw}, and the latter can be recovered in the (singular) limit
when $g $ converges to the Dirac mass at zero. An interesting special case of \eqref{GPLAW}
is obtained by choosing
$$g(s)= \frac{1}{\varpi \varsigma}\e^{-\frac{s}{\varpi\varsigma}}$$
where $\varsigma$ is a positive parameter.
In this way, one gets from \eqref{GPLAW} the so-called
Maxwell-Cattaneo law \cite{catta}, that is
\begin{equation}
\label{cattalaw}
\varsigma \varpi p_t + p= - \varpi \vartheta_x.
\end{equation}
Note that \eqref{cattalaw}
reduces to \eqref{foulaw} in the limit situation when $\varsigma=0$.

On the basis of these motivations, there has been a lot of activity directed towards the study
of PDE models where the parabolic Fourier law is replaced by a hyperbolic one. To the best of our knowledge,
the first results concerning the stability properties of Bresse and Timoshenko systems
with hyperbolic heat conduction have been obtained in the influential
papers \cite{SR,SJR}, dealing with the thermoelastic Timoshenko system with temperature obeying
the Maxwell-Cattaneo law. The analysis has been extended to the Timoshenko
system with Gurtin-Pipkin law in \cite{TIM}.
More recently, the Bresse system with Maxwell-Cattaneo law
has been studied in~\cite{SARE}. Further papers
treating ``reduced Bresse models" where one temperature is neglected
have appeared in the literature (see Section~\ref{coco} for details),
in addition to other articles where different damping mechanisms are analyzed.
The number of contributions on the subject is rather big, and
a comprehensive overview looks prohibitive. Disregarding papers
dealing exclusively with the Timoshenko system, we
may cite \cite{ALA,CS,FADE,FATM,FATRIV,GKA,GK,KH,MAMO,MEHA,RINA,SASO,SA2,SMRF,WE},
but the list is not exhaustive. Still, to the best of our knowledge, the
stability properties of the Bresse system with Gurtin-Pipkin law have not been investigated so far.
Moreover, the analyses in \cite{SARE,TIM} only deal with the exponential stability, and
no decay rates have been established when exponential stability does not occur.
The aim of the present paper is to fill these gaps.

\subsection{Main results}\label{infmain} We now describe briefly and informally our main results, postponing
the rigorous statements to the forthcoming Sections \ref{rigsecBGP}-\ref{rigsetTIMGP}.

\subsection*{I}
First, we study the Bresse-Gurtin-Pipkin (BGP) system
\begin{equation}
\label{bresse0}
\begin{cases}
\rho_1 \varphi_{tt} -k(\varphi_x +\psi +l w)_x - l k_0(w_x - l\varphi) + l\gamma \xi= 0,\\\noalign{\vskip3mm}
\rho_2 \psi_{tt} -b\psi_{xx} +k(\varphi_x +\psi +lw) +\gamma\vartheta_x= 0,\\\noalign{\vskip3mm}
\rho_1 w_{tt}  -k_0(w_x - l\varphi)_x + lk(\varphi_x +\psi +lw) +\gamma \xi_x= 0,\\\noalign{\vskip2.2mm}
\displaystyle
\rho_3 \vartheta_t - \varpi\int_0^\infty g(s)\vartheta_{xx}(t-s)\d s + \gamma\psi_{xt}=0,\\\noalign{\vskip1.1mm}
\displaystyle
\rho_3 \xi_t - \varpi\int_0^\infty h(s)\xi_{xx}(t-s)\d s + \gamma(w_{xt} - l \varphi_t)=0.\\
\end{cases}
\end{equation}
The model is complemented with the Dirichlet boundary conditions for the variables $\varphi,\vartheta,\xi$
\begin{equation}
\label{BC1}
\varphi(0,t) = \varphi(\ell,t) = \vartheta(0,t) = \vartheta(\ell,t) = \xi(0,t) = \xi(\ell,t) = 0,
\end{equation}
the Neumann boundary conditions for the variables $\psi,w$
\begin{equation}
\label{BC2}
\psi_x(0,t) = \psi_x(\ell,t) = w_x(0,t) = w_x(\ell,t)=0,
\end{equation}
and the appropriate initial data. In particular,
$\vartheta$ and $\xi$ are supposed to be known for negative times, where they need not solve the equations.
The convolution kernels $g$ and $h$ are nonnegative bounded convex summable functions on $[0,\infty)$, both
of unitary total mass and subjected to
some additional properties that will be specified in Subsection \ref{assmemker}.
Exploiting the so-called history framework devised
by C.M. Dafermos \cite{DAF}, system \eqref{bresse0} can be shown to generate
a solution semigroup $S(t)$ on an appropriate phase space.
Introducing the two stability numbers
\begin{align*}
&\chi_g = \Big(\frac{\rho_3}{\varpi g(0)} - \frac{\rho_1}{k} \Big)\Big(b - \frac{k\rho_2}{\rho_1} \Big)
+\frac{\gamma^2}{\varpi g(0)},\\
\noalign{\vskip0.7mm}
&\chi_h = \Big(\frac{\rho_3}{\varpi h(0)} - \frac{\rho_1}{k} \Big)\big(k_0 - k \big)+\frac{\gamma^2}{\varpi h(0)},
\end{align*}
we prove that the semigroup $S(t)$ is exponentially stable if and only if
\begin{equation}
\label{nsBGP}
\chi_g\hspace{0.4mm} \chi_h=0.
\end{equation}
Moreover, when $\chi_g\hspace{0.4mm} \chi_h\neq0$, we show that $S(t)$ is (semiuniformly) polynomially stable
with optimal decay rate~$\sqrt{t}$.

\subsection*{II}
Second, we analyze the Bresse-Maxwell-Cattaneo (BMC) system
\begin{equation}
\label{bressecatta}
\begin{cases}
\rho_1 \varphi_{tt} -k(\varphi_x +\psi +l w)_x - l k_0(w_x - l\varphi) + l\gamma \xi= 0,\\\noalign{\vskip0.5mm}
\rho_2 \psi_{tt} -b\psi_{xx} +k(\varphi_x +\psi +lw) +\gamma\vartheta_x= 0,\\\noalign{\vskip0.5mm}
\rho_1 w_{tt}  -k_0(w_x - l\varphi)_x + lk(\varphi_x +\psi +lw) +\gamma \xi_x= 0,\\\noalign{\vskip0.5mm}
\rho_3 \vartheta_t  + p_x+ \gamma\psi_{xt}=0,\\\noalign{\vskip0.5mm}
\varsigma \varpi p_t + p + \varpi \vartheta_x=0,\\\noalign{\vskip0.5mm}
\rho_3 \xi_t + q_x + \gamma(w_{xt} - l \varphi_t)=0,\\\noalign{\vskip0.5mm}
\tau \varpi q_t + q + \varpi \xi_x=0,
\end{cases}
\end{equation}
complemented with the boundary conditions \eqref{BC1}-\eqref{BC2} and the appropriate initial data.
The unknowns $p$ and $q$ represent the heat-flux variables, and
the constants $\varsigma$ and $\tau$ are strictly positive.
In \cite{SARE}, the BMC system \eqref{bressecatta}
is shown to generate a solution semigroup $T(t)$ on the natural phase space. In the same
article, the authors introduced the stability numbers
\begin{align*}
\chi_\varsigma = \Big(\varsigma\rho_3 - \frac{\rho_1}{k} \Big)\Big(b - \frac{k\rho_2}{\rho_1} \Big)
+\gamma^2\varsigma \, \and \, \chi_\tau = \Big(\tau \rho_3 - \frac{\rho_1}{k} \Big)\big(k_0 - k \big)+\gamma^2\tau,
\end{align*}
and proved that the semigroup $T(t)$ is exponentially stable when
\begin{equation}
\label{nsBMC}
\chi_\varsigma\hspace{0.4mm} \chi_\tau=0.
\end{equation}
Under additional restrictions on the coefficients, they also showed that
\eqref{nsBMC} is necessary for exponential stability.
Here, we complete the analysis of \cite{SARE} proving that \eqref{nsBMC} is indeed necessary and sufficient
for exponential stability. If $\chi_\varsigma\hspace{0.4mm} \chi_\tau\neq0$, we also demonstrate
that $T(t)$ is (semiuniformly) polynomially stable with optimal decay rate $\sqrt{t}$.

\subsection*{III}
Finally, we consider the Timoshenko-Gurtin-Pipkin (TGP) system
\begin{equation}
\label{TIMGP}
\begin{cases}
\rho_1 \varphi_{tt} -k(\varphi_x +\psi)_x = 0,\\
\noalign{\vskip2mm}
\rho_2 \psi_{tt} -b\psi_{xx} +k(\varphi_x +\psi) +\gamma\vartheta_x= 0,\\
\noalign{\vskip1mm}
\displaystyle
\rho_3 \vartheta_t  - \varpi\int_0^\infty g(s)\vartheta_{xx}(t-s)\d s + \gamma\psi_{xt}=0,
\end{cases}
\end{equation}
with the Dirichlet-Neumann-Dirichlet boundary conditions
$$
\varphi(0,t) = \varphi(\ell,t) = \psi_x(0,t) = \psi_x(\ell,t)=\vartheta(0,t) = \vartheta(\ell,t) =0
$$
and the appropriate initial data. Such a model can be obtained from \eqref{bresse0} in the limit
case $l=0$ (cf.\ Subsection \ref{PRE}). The exponential stability of the solution semigroup $U(t)$
associated to the TGP system \eqref{TIMGP} in the Dafermos history framework
has been analyzed in \cite{TIM}, where it is shown
that $U(t)$ is exponentially stable if and only if $\chi_g=0$. Here, we prove that when $\chi_g\neq0$
the semigroup $U(t)$ is (semiuniformly) polynomially stable
with optimal decay rate $\sqrt{t}$.

\begin{remark}
As anticipated in the Abstract, it is interesting to observe that \eqref{nsBGP}
is compatible with \eqref{phydef}. Such a feature has been already pointed out in
\cite{SARE} regarding \eqref{nsBMC}, which is compatible with \eqref{phydef} as well.
As already noticed, this does not happen when the Fourier law is employed, since condition \eqref{nscBF}
is not compatible with \eqref{phydef}.
\end{remark}

\begin{remark}
We shall not discuss the systems of second-order thermoelasticity
with Maxwell-Cattaneo or Gurtin-Pipkin laws, since
they have been already studied in the literature and their stability properties
are well-understood. In particular, exponential stability is the general rule here,
independently of the values of the structural constants (see e.g.\ \cite{PV,RAM2}).
\end{remark}

\subsection*{Plan of the paper} In the next Section \ref{coco}, we make some comments on
the results described so far, and we compare them
with some previous achievements on related models. We also mention few possible extensions
of our analysis. The short Section \ref{not} is devoted to the notation. In the subsequent
Sections~\ref{rigsecBGP}-\ref{rigsetTIMGP} we state
rigorously our main results, whose proofs are carried out
in the remaining Sections~\ref{sceres}-\ref{sezfinale}.


\section{Comments, Comparisons and Extensions}
\label{coco}

\subsection*{I}
The BMC system can be seen as a ``particular case" of the BGP one.
Indeed, taking in~\eqref{bresse0} the exponential kernels
\begin{equation}
\label{expchoice}
g_\varsigma(s) = \frac{1}{\varpi \varsigma}\e^{-\frac{s}{\varpi\varsigma}}
\and h_\tau(s) = \frac{1}{\varpi \tau}\e^{-\frac{s}{\varpi\tau}},
\end{equation}
it is readily seen that the heat-flux variables
\begin{align*}
&p(t) = -\varpi \int_{0}^\infty g_\varsigma(s)\vartheta_x(t-s) \d s,\\\noalign{\vskip0.7mm}
&q(t) = -\varpi \int_{0}^\infty h_\tau(s)\xi_x(t-s) \d s,
\end{align*}
satisfy the 5\textsuperscript{th} and the 7\textsuperscript{th} equation of \eqref{bressecatta}. Moreover,
since $\chi_\varsigma=\chi_{g_\varsigma}$ and $\chi_\tau=\chi_{h_\tau}$, the exponential stability
condition~\eqref{nsBMC} is formally recovered from \eqref{nsBGP}.
Actually, using the techniques of \cite[Section~8]{TIM}, one can prove
rigorously that the semigroup $S(t)$ corresponding to the choice \eqref{expchoice}
is exponentially stable if and only if the same does the semigroup $T(t)$.
Adapting the same techniques, one can also prove that if
$S(t)$ with the choice \eqref{expchoice} is polynomially stable with decay rate $\sqrt{t}$ then
the same does $T(t)$. In addition, the calculations
in the proof of the optimality of the decay rate of $S(t)$ can be easily adapted
to show the optimality of the decay rate of $T(t)$. We refer to
Section \ref{profcatta} of the present paper for details.

The same philosophy can be pursued in the study of the stability properties of
the Timoshenko-Maxwell-Cattaneo (TMC) system
$$
\begin{cases}
\rho_1 \varphi_{tt} -k(\varphi_x +\psi)_x = 0,\\
\rho_2 \psi_{tt} -b\psi_{xx} +k(\varphi_x +\psi) +\gamma\vartheta_x= 0,\\
\rho_3 \vartheta_t + p_x +\gamma\psi_{xt}=0,\\
\varsigma \varpi p_t + p + \varpi\vartheta_x = 0.
\end{cases}
$$
As shown in \cite{SR,SJR}, the solution
semigroup associated to the TMC system is exponentially stable if and only if
$\chi_\varsigma=0$, and polynomially stable with optimal decay rate $\sqrt{t}$ when $\chi_\varsigma\neq0$.
In \cite{TIM} it is proved that the semigroup $U(t)$ associated to the TGP system
with the choice $g=g_\varsigma$
is exponentially stable if and only if the same does the semigroup associated to the TMC system.
Moreover, arguing as in Section \ref{profcatta} of the present paper, one can prove
that if $U(t)$ with the choice $g=g_\varsigma$ is polynomially stable with optimal decay rate $\sqrt{t}$ then
the same does the semigroup associated to the TMC system.
Hence, both the exponential and the polynomial stability results
obtained in \cite{SR,SJR} are recovered.
We refrain from providing detailed proofs here, leaving them to the interested reader.

\subsection*{II} A key feature of the models considered in this work is that
the only source of dissipation is given by the thermal effects.
It goes without saying that other thermal damping are possible.
For instance, one can neglect the temperature $\xi$ and consider
the ``reduced BGP system"
\begin{equation}
\label{redbnew}
\begin{cases}
\rho_1 \varphi_{tt} -k(\varphi_x +\psi +l w)_x - l k_0(w_x - l\varphi) = 0,\\\noalign{\vskip3mm}
\rho_2 \psi_{tt} -b\psi_{xx} +k(\varphi_x +\psi +lw) +\gamma\vartheta_x= 0,\\\noalign{\vskip3mm}
\rho_1 w_{tt}  -k_0(w_x - l\varphi)_x + lk(\varphi_x +\psi +lw)= 0,\\\noalign{\vskip2.2mm}
\displaystyle
\rho_3 \vartheta_t - \varpi\int_0^\infty g(s)\vartheta_{xx}(t-s)\d s + \gamma\psi_{xt}=0.
\end{cases}
\end{equation}
This model has been analyzed in \cite{bresse-dello}, where it is shown that exponential
stability occurs if and only if
\begin{equation}
\label{nscbressedelloprimo}
\chi_g=0\, \and\, \chi_1=0.
\end{equation}
On the other hand, one can neglect the temperature $\vartheta$ and consider
the reduced system
\begin{equation}
\label{redbnewbis}
\begin{cases}
\rho_1 \varphi_{tt} -k(\varphi_x +\psi +l w)_x - l k_0(w_x - l\varphi) + l\gamma \xi= 0,\\\noalign{\vskip3mm}
\rho_2 \psi_{tt} -b\psi_{xx} +k(\varphi_x +\psi +lw) = 0,\\\noalign{\vskip3mm}
\rho_1 w_{tt}  -k_0(w_x - l\varphi)_x + lk(\varphi_x +\psi +lw) +\gamma \xi_x= 0,\\\noalign{\vskip2.2mm}
\displaystyle
\rho_3 \xi_t - \varpi\int_0^\infty h(s)\xi_{xx}(t-s)\d s + \gamma(w_{xt} - l \varphi_t)=0.
\end{cases}
\end{equation}
To the best of our knowledge, the latter has not been studied so far in the literature, but we
conjecture that the necessary and sufficient condition for its exponential stability is
\begin{equation}
\label{nscbressenew}
\chi_h=0\, \and\, \chi_0=0.
\end{equation}
In the light of the previous discussions, this insight becomes clear if one considers the
Maxwell-Cattaneo version of \eqref{redbnewbis}
$$
\begin{cases}
\rho_1 \varphi_{tt} -k(\varphi_x +\psi +l w)_x - l k_0(w_x - l\varphi) + l\gamma \xi= 0,\\\noalign{\vskip0.5mm}
\rho_2 \psi_{tt} -b\psi_{xx} +k(\varphi_x +\psi +lw) = 0,\\\noalign{\vskip0.5mm}
\rho_1 w_{tt}  -k_0(w_x - l\varphi)_x + lk(\varphi_x +\psi +lw) +\gamma \xi_x= 0,\\\noalign{\vskip0.5mm}
\rho_3 \xi_t + q_x + \gamma(w_{xt} - l \varphi_t)=0,\\\noalign{\vskip0.5mm}
\tau \varpi q_t + q + \varpi \xi_x=0,
\end{cases}
$$
where, as shown in \cite{SAREzamp}, exponential stability occurs if and only if
$$
\chi_\tau=0\, \and\, \chi_0=0.
$$
Note that both conditions \eqref{nscbressedelloprimo}
and \eqref{nscbressenew} are incompatible with \eqref{phydef}.

One might think that the exponential stability analysis of \eqref{redbnew}
and \eqref{redbnewbis} is more challenging with respect to the one of \eqref{bresse0},
since the latter has a stronger damping mechanism contributed by two equations.
Such a thought appears unfounded. Indeed, as for the BF system, in \eqref{bresse0}
the variables $\psi$ and $w$ are effectively damped
but the variable $\varphi$ is only
indirectly damped (cf.\ the discussion at the end
of Subsection \ref{stabbf}).
In order to achieve the exponential stability,
one has to ``enucleate" the damping contribution of $\vartheta$ using the equality $\chi_g=0$
to stabilize exponentially the variable $\varphi$,
and do the same for the damping contribution of $\xi$ using the equality $\chi_h=0$. As in the BF system,
these conditions provide appropriate cancellations of some higher-order terms that pop up in the estimates.
The exponential stability of~\eqref{redbnew} requires basically
half of the job:\ in \eqref{redbnew} only the variable $\psi$ is effectively damped,
so that one has to use the equality $\chi_g=0$ to stabilize exponentially
the variable $\varphi$ and the equality $\chi_1=0$ to stabilize exponentially the variable $w$, but
the condition $\chi_1=0$ is quite easy to exploit (see \cite[Section~6]{bresse-dello}).
Similar remarks apply in the analysis of~\eqref{redbnewbis}.

\subsection*{III}
The methodology in this paper is based on resolvent estimates
combined with the abstract results of L. Gearhart $\&$ J. Pr\"{u}ss, A. Borichev $\&$ Y. Tomilov
and  C.J.K. Batty $\&$ T. Duyckaerts \cite{BattyDuy,BT,Ge,Pru}.
Although this approach is not new, the analysis of the present
work presents some peculiar elements. We highlight the following aspects.

\begin{itemize}
\item In order to show the optimality of the polynomial decay rate,
one needs to exploit sharp lower resolvent estimates which
require the use of a quantified version of the Riemann-Lebesgue lemma recently obtained in \cite{DLP}.

\smallskip
\item
We provide a theoretical method to show that the polynomial
decay rate $\sqrt{t}$ of the BGP system
with the choice \eqref{expchoice} implies automatically the same polynomial decay rate of the
BMC system. It is very likely that such a method can be successfully applied to other
models too.

\smallskip
\item
The study of the BGP and the TGP systems presents some difficulties connected
with a structural lack of compactness which complicates the spectral analysis of the
infinitesimal generators (see Remark \ref{remcompt} for details).

\smallskip
\item The complexity of the BGP system
(three wave equations coupled with two integrodifferential hyperbolic equations) requires
a rather heavy technical effort when performing the resolvent estimates.

\end{itemize}

\subsection*{IV}
In the limit case $\varsigma=\tau=0$,
the BMC system (formally) reduces to the BF system, and condition \eqref{nsBMC}
boils down to \eqref{nscBF}.
Actually, as noticed in \cite{bresse-dello,TIM}, the BF system can be recovered directly from the BGP system by means
of an appropriate singular
limit procedure in which the kernels
approach the Dirac mass at zero $\delta_0$. More precisely, for $\eps>0$, let us set
\begin{equation}\label{resca}
g_\eps (s) = \frac{1}{\eps} g \left(\frac{s}{\eps} \right) \and
h_\eps (s) = \frac{1}{\eps} h \left(\frac{s}{\eps} \right).
\end{equation}
Since $g_\eps\to \delta_0$ and $h_\eps\to\delta_0$ in the distributional sense as $\eps\to0$,
system~\eqref{bresse0} with the choice
$g=g_{\eps}$ and $h=h_{\eps}$ (formally) boils down to \eqref{bresse-fou} in the limit $\eps\to0$.
Note also that
$$\chi_{g_\eps}\to-\frac{\rho_1}{k}\chi_0 \and
\chi_{h_\eps}\to-\frac{\rho_1}{k}\chi_1$$ when $\eps\to0$, so that condition \eqref{nsBGP} reduces to \eqref{nscBF}.
The same phenomenon appears in the passage from the TGP system \eqref{TIMGP}
to the TF system \eqref{TF} (see \cite{TIM} for details).
A rigorous proof of the convergence of solutions to equations with memory
through the ones of the limit equation when the kernel collapses into a Dirac mass
has been given in \cite{AMNE}.

\subsection*{V} As pointed out in \cite{bresse-dello,TIM}, once the exponential stability properties
of the BGP system are known it is possible to characterize the ones
of the Bresse-Coleman-Gurtin system
\begin{equation}
\label{bresseCG}
\begin{cases}
\rho_1 \varphi_{tt} -k(\varphi_x +\psi +l w)_x - l k_0(w_x - l\varphi) + l\gamma \xi= 0,\\\noalign{\vskip3mm}
\rho_2 \psi_{tt} -b\psi_{xx} +k(\varphi_x +\psi +lw) +\gamma\vartheta_x= 0,\\\noalign{\vskip3.3mm}
\rho_1 w_{tt}  -k_0(w_x - l\varphi)_x + lk(\varphi_x +\psi +lw) +\gamma \xi_x= 0,\\\noalign{\vskip2.2mm}
\displaystyle
\rho_3 \vartheta_t -\varpi(1-m)\vartheta_{xx} - \varpi m
\int_0^\infty g(s)\vartheta_{xx}(t-s)\d s + \gamma\psi_{xt}=0,\\\noalign{\vskip1.1mm}
\displaystyle
\rho_3 \xi_t -\varpi(1-m)\xi_{xx} -\varpi m \int_0^\infty h(s)\xi_{xx}(t-s)\d s + \gamma(w_{xt} - l \varphi_t)=0.\\
\end{cases}
\end{equation}
In the model above, $m\in (0,1)$ is a fixed parameter and the
temperatures obey the parabolic-hyperbolic law introduced by B.D. Coleman and M.E. Gurtin in
\cite{CGu}. The limit cases $m=0$ and $m=1$ correspond to the BF system \eqref{bresse-fou} and the BGP system
\eqref{bresse0},
respectively. The solution semigroup
associated to \eqref{bresseCG} in the Dafermos history framework is exponentially stable if and only if
$\chi_0 \hspace{0.4mm} \chi_1=0$, meaning that the parabolic character
prevails on the hyperbolic one. To see that, following the procedure introduced in \cite{TIM}, let us set
for $\eps>0$
$$
g_{\eps} (s) = \frac{1-m}{\eps} g \left(\frac{s}{\eps} \right) + m g(s)\and
h_{\eps} (s) = \frac{1-m}{\eps} h \left(\frac{s}{\eps} \right) + m h(s).
$$
Since $g_\eps\to (1-m)\delta_0 + m g$ and $h_\eps\to(1-m)\delta_0 + m h$
in the distributional sense when $\eps\to0$,
system~\eqref{bresse0} with the choice
$g=g_{\eps}$ and $h=h_{\eps}$ (formally) boils down to \eqref{bresseCG} as $\eps\to0$.
At the same time, we have the convergence
$$\chi_{g_\eps}\to-\frac{\rho_1}{k}\chi_0 \and
\chi_{h_\eps}\to-\frac{\rho_1}{k}\chi_1$$ for $\eps\to0$, so that
condition \eqref{nsBGP} reduces to $\chi_0 \hspace{0.4mm} \chi_1=0$.

\subsection*{VI} As mentioned in the last section of \cite{SARE}, it is possible
to consider ``mixed Bresse models" where the temperatures obey two different thermal laws.
For instance, one can study a system in which $\vartheta$
satisfies the Gurtin-Pipkin law and $\xi$ the Maxwell-Cattaneo one (and vice versa),
or in which $\vartheta$
satisfies the Gurtin-Pipkin law and $\xi$ the Fourier one (and vice versa), and other combinations
(including the Coleman-Gurtin law).
All the exponential stability conditions for these systems can be derived from \eqref{nsBGP}
by choosing appropriate kernels as in \eqref{expchoice} or through appropriate
singular limit procedures as above. For instance,
if $\vartheta$ satisfies the Gurtin-Pipkin law and $\xi$ the Maxwell-Cattaneo one
the exponential stability condition reads $\chi_g\hspace{0.4mm} \chi_\tau=0$, while if $\vartheta$
satisfies the Gurtin-Pipkin law and $\xi$ the Fourier one
it reads $\chi_g\hspace{0.4mm} \chi_1=0$, and so on.


\section{Notation}
\label{not}

\noindent
The notation is mostly standard throughout. In particular, $\R^+\doteq(0,\infty)$
denotes the positive half-line,
$\mathbb{N}=1,2,3,\ldots$ the set of positive integers, and
$\i\R$ the imaginary axis in the complex plane. The ``Big O" and
``Little-o" notations for functions or sequences have the standard meaning.
Given a closed linear operator $\mathsf{L}$ acting on a complex
Hilbert space, we denote the domain by $\D(\mathsf{L})$, the resolvent set
by $\varrho(\mathsf{L})$ and the spectrum by $\sigma(\mathsf{L})$.
The symbols $L^2, H^1,H_0^1,H^2$ indicate the usual complex
Lebesgue and Sobolev spaces on the interval $(0,\ell)$, while
$\langle\cdot,\cdot\rangle$ and $\|\cdot\|$ stand for the
standard inner product and norm on $L^2$.
Since no confusion can occur, the symbol $\|\cdot\|$ will be also used to
denote the operator norm. We will also work with
the Hilbert spaces of zero-mean functions
$$
L^2_*=\big\{ f\in L^2 : \int_0^\ell f(x) \d x = 0\big\}\,\,
\and\,\,  H^1_* = H^1\cap L^2_*,
$$
the latter equipped with the gradient norm (due to the Poincar\'e inequality).
Along the paper, we routinely employ the Young, H\"older
and Poincar\'e inequalities without explicit mention.


\section{Rigorous Statements for the BGP System}
\label{rigsecBGP}

\subsection{Assumptions on the kernels}
\label{assmemker}

The convolution kernels $g$ and $h$
are nonnegative bounded convex
summable functions, both of unitary total mass
and having the explicit form
$$
g(s) = \int_s^\infty \mu(r) \d r \, \and h(s) = \int_s^\infty \nu(r) \d r,
$$
where $\mu,\nu:\R^+\to\R^+$, called memory kernels, are
nonincreasing absolutely continuous functions. In particular,
$\mu$ and $\nu$ are summable with
$$\int_0^\infty \mu(r) \d r = g(0) \and \int_0^\infty \nu(r) \d r = h(0).
$$
We also require that $\mu$ and
$\nu$ are bounded about zero, namely
\begin{align*}
\mu(0) \doteq \lim_{s\to 0} \mu(s)<\infty\and
\nu(0) \doteq \lim_{s\to 0} \nu(s) <\infty.
\end{align*}
Finally, we assume the so-called Dafermos conditions
\begin{align}
\label{assnucleo1}
&\mu'(s) + \delta_\mu\,\mu(s) \leq 0,\\
\label{assnucleo2}
&\nu'(s) + \delta_\nu\,\nu(s) \leq 0,
\end{align}
for some $\delta_\mu,\delta_\nu>0$ and almost every $s>0$.

\subsection{Memory spaces}\label{memspace} We introduce the
so-called memory spaces
$$\M = L^2_\mu(\R^+; H^1_0) \and \N = L^2_\nu(\R^+; H^1_0)$$
of square summable $H_0^1$-valued functions on $\R^+$ with respect to the measures $\mu(s)\d s$
and $\nu(s) \d s$, respectively,
endowed with the inner products
\begin{align*}
&\langle\eta_1,\eta_2\rangle_\M=\int_0^\infty \mu(s) \langle\eta_{1x}(s),\eta_{2x}(s)\rangle \d s,\\
&\langle\xi_1,\xi_2\rangle_\N=\int_0^\infty \nu(s) \langle\xi_{1x}(s),\xi_{2x}(s)\rangle \d s.
\end{align*}
The induced norms will be denoted by $\|\cdot\|_\M$ and $\|\cdot\|_\N$.
Moreover, we consider the infinitesimal generator of the right-translation semigroup on $\M$,
that is, the operator
$$T \eta=-\eta' \qquad\, \text{with}\qquad\, \D(T)=\big\{\eta\in{\M}:\eta'\in\M,\,\,
\lim_{s\to 0}\|\eta_x(s)\|=0\big\},$$
where $\eta'$ stands for the weak derivative with respect to the variable $s\in \R^+$.
We will also work with the infinitesimal generator of the right-translation semigroup on $\N$,
denoted with the same symbol $T$ and defined in the same way.

\subsection{Extended memory space}
We introduce the so-called extended memory space
$$
\H = H_0^1 \times L^2 \times H_*^1 \times L_*^2\times H_*^1 \times L_*^2
\times L^2 \times \M \times L^2 \times \N
$$
equipped with the norm
\begin{align*}
\|u\|_\H^2
&= k\|\varphi_x+\psi+lw\|^2 + \rho_1\|\Phi\|^2 +b\|\psi_x\|^2+\rho_2\|\Psi\|^2 +k_0\|w_x-l\varphi\|^2\\
&\quad +\rho_1\|W\|^2+ \rho_3\|\vartheta\|^2
+ \varpi \|\eta\|^2_\M + \rho_3 \|\xi\|^2  + \varpi\|\zeta\|^2_\N
\end{align*}
for every $u = (\varphi,\Phi,\psi,\Psi,w,W,\vartheta,\eta,\xi,\zeta)\in\H$.
The inner product associated to $\|\cdot\|_\H$ will be denoted by $\l \cdot, \cdot \r_\H$.
As customary in the analysis of Bresse systems, we are tacitly assuming that
\begin{equation}
\label{condition}
l\ell \neq n \pi,\quad\, \forall n \in \mathbb{N}.
\end{equation}
In fact, if \eqref{condition} is violated it is not difficult
to construct nonzero $u\in\H$ with $\|u\|_\H=0$.
Instead, when \eqref{condition} holds true,
$\|\cdot\|_\H$ becomes a norm on $\H$, equivalent to the standard
product norm
$$
\boldsymbol{|}u\boldsymbol{|}_\H^2 = \|\varphi_x\|^2 + \|\Phi\|^2 + \|\psi_x\|^2+ \|\Psi\|^2 + \|w_x\|^2
+ \|W\|^2+ \|\vartheta\|^2
+ \|\eta\|^2_\M + \|\xi\|^2  + \|\zeta\|^2_\N.
$$
In particular, there exists a structural constant $\mathfrak{c}>0$ such that
\begin{equation}
\label{equivnorm}
\mathfrak{c}\boldsymbol{|}u\boldsymbol{|}_\H \leq \|u\|_{\H} \leq
\frac1{\mathfrak{c}} \boldsymbol{|}u\boldsymbol{|}_\H,\quad\, \forall u \in \H.
\end{equation}
Indeed, the second inequality above is immediate, and within \eqref{condition} is not hard to check
that $\|\cdot\|_\H$ is a Banach norm on $\H$. Thus
\eqref{equivnorm} follows from the Open Mapping Theorem.
Along the paper, relation~\eqref{equivnorm} will be tacitly employed in several occasions.

\subsection{The semigroup}\label{secsemi}
We reformulate the BGP system \eqref{bresse0} making use
of the history framework of Dafermos \cite{DAF}. To this end, for $s>0$,
we consider the auxiliary variables
$$
\eta^t(x,s)=\int_{0}^s \vartheta(x,t-r)\d r \and \zeta^t(x,s)=\int_{0}^s \xi(x,t-r)\d r,$$
and we rewrite \eqref{bresse0} in the form
\begin{equation}
\label{bresse-rew}
\begin{cases}
\rho_1 \varphi_{tt} -k(\varphi_x +\psi +l w)_x - l k_0 (w_x - l\varphi) + l\gamma \xi= 0,\\\noalign{\vskip2.5mm}
\rho_2 \psi_{tt} -b\psi_{xx} +k(\varphi_x +\psi +lw) +\gamma\vartheta_x= 0,\\\noalign{\vskip2.9mm}
\rho_1 w_{tt}  -k_0(w_x - l\varphi)_x + l k(\varphi_x +\psi +lw) +\gamma \xi_x= 0,\\\noalign{\vskip0.8mm}
\displaystyle
\rho_3\vartheta_t -\varpi\int_0^\infty \mu (s)\eta_{xx}(s)\d s + \gamma \psi_{xt}=0,\\\noalign{\vskip0.7mm}
\eta_t= T \eta + \vartheta,\\
\displaystyle
\rho_3\xi_t -\varpi\int_0^\infty \nu(s)\zeta_{xx}(s)\d s + \gamma(w_{xt} - l \varphi_t)=0,\\\noalign{\vskip0.7mm}
\zeta_t = T \zeta + \xi.
\end{cases}
\end{equation}
Introducing the state vector
$u(t) = (\varphi(t),\Phi(t),\psi(t),\Psi(t),w(t),W(t),\vartheta(t),\eta^t,\xi(t),\zeta^t)\in\H,$
we view \eqref{bresse-rew} as the abstract ODE on $\H$
\begin{equation}
\label{refabs}
\ddt u(t) = \mathsf{A} u(t),
\end{equation}
where the linear operator $\mathsf{A}$ is defined as
$$
\mathsf{A}
\left(\begin{matrix}
\varphi\\\noalign{\vskip.5mm}
\Phi\\\noalign{\vskip.5mm}
\psi\\\noalign{\vskip.5mm}
\Psi\\\noalign{\vskip.5mm}
w\\\noalign{\vskip.5mm}
W\\\noalign{\vskip.5mm}
\vartheta\\\noalign{\vskip.5mm}
\eta\\\noalign{\vskip.5mm}
\xi\\\noalign{\vskip.5mm}
\zeta
\end{matrix}
\right)
=\left(
\begin{matrix}
\Phi\\
\frac{k}{\rho_1} (\varphi_x + \psi +lw)_x + \frac{lk_0}{\rho_1}(w_x-l\varphi)-\frac{l\gamma}{\rho_1} \xi\\
\Psi\\
\frac{b}{\rho_2}\psi_{xx} - \frac{k}{\rho_2}(\varphi_x+\psi+lw) - \frac{\gamma}{\rho_2}\vartheta_x\\
W\\
\frac{k_0}{\rho_1}(w_x-l\varphi)_x - \frac{lk}{\rho_1} (\varphi_x + \psi +lw)-\frac{\gamma}{\rho_1} \xi_x\\
\noalign{\vskip1mm}
\frac{\varpi}{\rho_3}\int_0^\infty \mu(s) \eta_{xx}(s)\d s - \frac{\gamma}{\rho_3}\Psi_x\\
\noalign{\vskip.5mm}
T\eta + \vartheta\\
\noalign{\vskip.5mm}
\frac{\varpi}{\rho_3}\int_0^\infty \nu(s)\zeta_{xx}(s)\d s - \frac{\gamma}{\rho_3}(W_{x} - l \Phi)\\
\noalign{\vskip.5mm}
T \zeta + \xi
\end{matrix}
\right)
$$
with domain
$$
\D(\mathsf{A}) = \left\{u\in\H \left|\,\,
\begin{matrix}
\varphi \in H^2\\
\Phi, \psi_x,w_x\in  H_0^1\\
\Psi, W \in H^1\\
\vartheta,\xi \in H_0^1\\
\int_0^\infty \mu(s)\eta(s)\d s,\int_0^\infty \nu(s)\zeta(s)\d s \in H^2\\
\eta,\zeta \in \D(T)\\
\end{matrix}\right.
\right\}.
$$
For every $u\in\D(\mathsf{A})$,
a straightforward computation entails
$$
\Re \l \mathsf{A} u , u\r_\H = \varpi \, \Re \l T \eta, \eta\r_{\M} +  \varpi\,\Re \l T \zeta, \zeta\r_{\N}.
$$
Moreover, as shown in \cite{Terreni}, we have the equality
$$
\Re \l T \eta, \eta\r_{\M} +  \Re \l T \zeta, \zeta\r_{\N}
=-\frac{1}{2}\big[\Gamma[\eta]  + \Gamma[\zeta] \big],
$$
where we set
$$
\Gamma[\eta] = \int_0^\infty -\mu'(s) \|\eta_x(s)\|^2 \d s \and
\Gamma[\zeta] = \int_0^\infty -\nu'(s) \|\zeta_x(s)\|^2 \d s.
$$
Since $\Gamma[\eta]$ and $\Gamma[\zeta]$ are nonnegative, we are led to
\begin{equation}
\label{dissip}
\Re \l \mathsf{A} u , u\r_\H =
-\frac{\varpi}{2}\big[\Gamma[\eta]  + \Gamma[\zeta] \big] \leq 0,
\end{equation}
meaning that $\mathsf{A}$ is dissipative. In addition, by means of standard techniques
based on the Lax-Milgram theorem, one can prove that the operator $1-\mathsf{A}$ is surjective
(see e.g. \cite{butanino,viscofraz} for details on the procedure in the context of equations with memory).
In particular, $\mathsf{A}$ is densely defined \cite[Theorem 4.6]{PAZY} and, due to the Lumer-Phillips theorem,
it generates a contraction semigroup
$$
S(t) = \e^{t \mathsf{A}} : \H \to \H.
$$
In particular, for every initial datum $u_0\in\H$, equation \eqref{refabs} admits a unique (mild) solution
$u$ given by $u(t)=S(t)u_0$. If $u_0\in\D(\mathsf{A})$,
then the solution $u$ is classical (see e.g.\ \cite{PAZY}).

\subsection{The results}
The rigorous statements of the stability results for the BGP system anticipated in the Introduction
read as follows ($\chi_g$ and $\chi_h$ have been defined in Subsection~\ref{infmain}).

\begin{theorem}
\label{EXP-STAB-TEO}
Assume that $\chi_g\hspace{0.3mm} \chi_h=0$. Then the semigroup $S(t)$ is exponentially stable, namely,
there exist two structural constants\hspace{0.2mm} $\omega>0$ and $K=K(\omega)\geq1$ such that
\begin{equation}
\label{defexp}
\|S(t)\| \leq K \e^{-\omega t}, \quad\, \forall t\geq0.
\end{equation}
\end{theorem}

\begin{theorem}
\label{POL-STAB-TEO}
The semigroup $S(t)$ is polynomially semiuniformly stable with decay rate $\sqrt{t}$, namely,
$0\in \varrho(\mathsf{A})$ and
there exists a structural constant $K>0$ such that
\begin{equation}
\label{condpoly}
\|S(t)\mathsf{A}^{-1}\| \leq \frac{K}{\sqrt{t}}, \quad\, \forall t>0.
\end{equation}
If in addition $\chi_g \hspace{0.3mm} \chi_h\neq0$, then such a decay rate is optimal, namely
\begin{equation}
\label{condpolyopt}
\limsup_{t\to\infty} \sqrt{t}\, \|S(t)\mathsf{A}^{-1}\| >0.
\end{equation}
\end{theorem}

It is readily seen that condition \eqref{condpoly} can be reformulated as
$$
\|S(t)u_0\|_\H \leq \frac{K}{\sqrt{t}}\|\mathsf{A} u_0\|_\H,\quad\, \forall t>0,\,\,\, \forall u_0\in\D(\mathsf{A}).
$$
Thus, for all $u_0\in \D(\mathsf{A})$, we have the convergence $S(t)u_0\to 0$ as $t\to\infty$.
Since $S(t)$ is a contraction semigroup, we conclude that

\begin{corollary}
The semigroup $S(t)$ is stable, namely, for every fixed $u_0\in\H$ we have
$$
\lim_{t\to\infty} \|S(t)u_0\|_{\H} = 0.
$$
\end{corollary}

When $\chi_g\hspace{0.3mm} \chi_h\neq0$, relation \eqref{condpolyopt} tells
that $S(t)$ cannot be exponentially stable.
Hence, we get

\begin{corollary}
The semigroup $S(t)$ is exponentially stable if and only if\hspace{0.2mm} $\chi_g\hspace{0.3mm} \chi_h=0$.
\end{corollary}


\section{Rigorous Statements for the BMC System}

\noindent
We begin by introducing the product space
$$
\V = H_0^1 \times L^2 \times H_*^1 \times L_*^2\times H_*^1 \times L_*^2
\times L^2 \times L_*^2 \times L^2 \times L_*^2
$$
equipped with the norm
\begin{align*}
\|v\|_\V^2
&= k\|\varphi_x+\psi+lw\|^2 + \rho_1\|\Phi\|^2 +b\|\psi_x\|^2+\rho_2\|\Psi\|^2 +k_0\|w_x-l\varphi\|^2\\
&\quad +\rho_1\|W\|^2+ \rho_3\|\vartheta\|^2
+ \varsigma\|p\|^2 + \rho_3 \|\xi\|^2  + \tau\|q\|^2
\end{align*}
for every
$v = (\varphi,\Phi,\psi,\Psi,w,W,\vartheta,p,\xi,q)\in\V$.
As before, we tacitly assume that \eqref{condition} is satisfied, so that
$\|\cdot\|_\V$ is a norm on $\V$, equivalent to the standard product norm.
Then, we view the BMC system \eqref{bressecatta} as the abstract ODE on $\V$
\begin{equation}
\label{refabsCATTA}
\ddt v(t) = \mathsf{B} v(t),
\end{equation}
where the linear operator $\mathsf{B}$ is defined as
$$
\mathsf{B}
\left(\begin{matrix}
\varphi\\\noalign{\vskip.5mm}
\Phi\\\noalign{\vskip.5mm}
\psi\\\noalign{\vskip.5mm}
\Psi\\\noalign{\vskip.5mm}
w\\\noalign{\vskip.5mm}
W\\\noalign{\vskip.5mm}
\vartheta\\\noalign{\vskip.5mm}
p\\\noalign{\vskip.5mm}
\xi\\\noalign{\vskip.5mm}
q
\end{matrix}
\right)
=\left(
\begin{matrix}
\Phi\\
\frac{k}{\rho_1} (\varphi_x + \psi +lw)_x + \frac{lk_0}{\rho_1}(w_x-l\varphi)-\frac{l\gamma}{\rho_1} \xi\\
\Psi\\
\frac{b}{\rho_2}\psi_{xx} - \frac{k}{\rho_2}(\varphi_x+\psi+lw) - \frac{\gamma}{\rho_2}\vartheta_x\\
W\\
\frac{k_0}{\rho_1}(w_x-l\varphi)_x - \frac{lk}{\rho_1} (\varphi_x + \psi +lw)-\frac{\gamma}{\rho_1} \xi_x\\
\noalign{\vskip1.3mm}
-\frac{1}{\rho_3}p_x - \frac{\gamma}{\rho_3}\Psi_x\\
\noalign{\vskip0.7mm}
-\frac{1}{\varsigma \varpi}p - \frac{1}{\varsigma}\vartheta_x\\
\noalign{\vskip0.7mm}
-\frac{1}{\rho_3}q_x - \frac{\gamma}{\rho_3}(W_{x} - l \Phi)\\
\noalign{\vskip0.7mm}
-\frac{1}{\tau \varpi}q - \frac{1}{\tau}\xi_x
\end{matrix}
\right)
$$
with domain
$$
\D(\mathsf{B}) = \left\{v\in\V \left|\,\,
\begin{matrix}
\varphi \in H^2\\
\Phi, \psi_x,w_x\in  H_0^1\\\noalign{\vskip0.5mm}
\Psi, W\in H^1\\
\vartheta,\xi \in H_0^1\\\noalign{\vskip0.3mm}
p,q \in H^1
\end{matrix}\right.
\right\}.
$$
According to \cite[Theorem 2.2]{SARE}, the operator $\mathsf{B}$ generates a contraction semigroup
$$T(t)=\e^{t \mathsf{B} }:\V\to\V.$$
In the same paper, the following are also shown ($\chi_\tau$ and $\chi_\varsigma$ have been
defined in Subsection~\ref{infmain}).

\begin{itemize}

\item The inclusion $\i\R\subset \varrho(\mathsf{B})$ holds.

\vspace{0.6mm}
\item The semigroup $T(t)$ is exponentially stable when
$\chi_\varsigma\hspace{0.4mm} \chi_\tau=0$.

\vspace{0.6mm}
\item The semigroup $T(t)$ is not exponentially stable when
$\chi_\varsigma\hspace{0.4mm} \chi_\tau\neq0$ and the coefficients fulfill additional
constraints (see \cite[pp.\ 3594-3595]{SARE}).
\end{itemize}

\noindent
As mentioned in the Introduction, our result completes such an analysis.

\begin{theorem}
\label{POL-STAB-TEO-CATTA}
The semigroup $T(t)$ is polynomially semiuniformly stable with decay rate $\sqrt{t}$, namely,
there exists a structural constant $K>0$ such that
\begin{equation}
\label{polycattaott}
\|T(t)\mathsf{B}^{-1}\| \leq \frac{K}{\sqrt{t}}, \quad\, \forall t>0.
\end{equation}
If in addition $\chi_\varsigma \hspace{0.3mm} \chi_\tau\neq0$, then such a decay rate is optimal, namely
\begin{equation}
\label{polycattaottimal}
\limsup_{t\to\infty} \sqrt{t}\, \|T(t)\mathsf{B}^{-1}\| >0.
\end{equation}
\end{theorem}

Similarly to the BGP system, we also have

\begin{corollary}
The semigroup $T(t)$ is exponentially stable (if and) only if $\chi_\varsigma \hspace{0.3mm}  \chi_\tau = 0$.
\end{corollary}


\section{Rigorous Statements for the TGP System}
\label{rigsetTIMGP}

\noindent
First, we consider the product space
$$
\Z = H_0^1 \times L^2 \times H_*^1 \times L_*^2
\times L^2 \times \M
$$
equipped with the norm (equivalent to the standard product norm)
\begin{align*}
\|z\|_\Z^2
&= k\|\varphi_x+\psi\|^2 + \rho_1\|\Phi\|^2 +b\|\psi_x\|^2+\rho_2\|\Psi\|^2 +
\rho_3\|\vartheta\|^2 + \varpi\|\eta\|^2_\M
\end{align*}
for every
$z=(\varphi,\Phi,\psi,\Psi,\vartheta,\eta)\in \Z$.
The memory space $\M$ and its norm $\|\cdot\|_{\M}$ have been defined in Subsection \ref{memspace}.
The inner product associated to $\|\cdot\|_\Z$ will be denoted by $\l \cdot, \cdot \r_\Z$.
Next, as in the BGP system, we
introduce the auxiliary variable
$$
\eta^t(x,s)=\int_{0}^s \vartheta(x,t-r)\d r,\quad\, s>0,
$$
and we rewrite the TGP system \eqref{TIMGP} in the form
\begin{equation}
\label{tim-rew}
\begin{cases}
\rho_1 \varphi_{tt} -k(\varphi_x +\psi)_x = 0,\\\noalign{\vskip2mm}
\rho_2 \psi_{tt} -b\psi_{xx} +k(\varphi_x +\psi) +\gamma\vartheta_x= 0,\\\noalign{\vskip0.7mm}
\displaystyle
\rho_3\vartheta_t -\varpi\int_0^\infty \mu (s)\eta_{xx}(s)\d s + \gamma \psi_{xt}=0,\\
\eta_t= T \eta + \vartheta.
\end{cases}
\end{equation}
The memory kernel $\mu$ has been defined in Subsection \ref{assmemker} and fulfills the properties stated therein
(in particular, $\mu$ is bounded about zero and complies with \eqref{assnucleo1}),
while the operator $T$ has been defined in Subsection \ref{memspace}.
As customary, we view \eqref{tim-rew} as the abstract ODE on $\Z$
$$
\ddt z(t) = \mathsf{C} z(t),
$$
where the linear operator $\mathsf{C}$ reads
$$
\mathsf{C}
\left(\begin{matrix}
\varphi\\\noalign{\vskip.7mm}
\Phi\\
\psi\\\noalign{\vskip.5mm}
\Psi\\\noalign{\vskip.5mm}
\vartheta\\\noalign{\vskip.5mm}
\eta
\end{matrix}
\right)
=\left(
\begin{matrix}
\Phi\\
\frac{k}{\rho_1} (\varphi_x + \psi)_x\\
\Psi\\
\frac{b}{\rho_2}\psi_{xx} - \frac{k}{\rho_2}(\varphi_x+\psi) - \frac{\gamma}{\rho_2}\vartheta_x\\
\noalign{\vskip1.5mm}
\frac{\varpi}{\rho_3}\int_0^\infty \mu(s) \eta_{xx}(s) \d s - \frac{\gamma}{\rho_3}\Psi_x\\
\noalign{\vskip.5mm}
T\eta + \vartheta
\end{matrix}
\right)
$$
with domain
$$
\D(\mathsf{C}) =\left\{z\in\Z\left|\,\,
\begin{matrix}
\varphi\in H^2\\
\Phi,\psi_x\in H_0^1\\
\Psi \in H^1\\
\vartheta \in H_0^1\\
\int_0^\infty \mu(s)\eta(s)\d s \in H^2\\
\eta \in \D(T)
\end{matrix}\right.
\right\}.
$$
According to \cite[Theorem 3]{TIM}, the operator $\mathsf{C}$ generates a contraction semigroup
$$U(t)=\e^{t \mathsf{C}}:\Z\to\Z.$$
As mentioned in the Introduction, such a semigroup
is exponentially stable if and only if $\chi_g=0$. Our result reads as follows.

\begin{theorem}
\label{POL-STAB-TEO-TIM}
The semigroup $U(t)$ is polynomially semiuniformly stable with decay rate $\sqrt{t}$, namely,
$0\in\varrho(\mathsf{C})$ and
there exists a structural constant $K>0$ such that
\begin{equation}
\label{polytimott}
\|U(t)\mathsf{C}^{-1}\| \leq \frac{K}{\sqrt{t}}, \quad\, \forall t>0.
\end{equation}
If in addition $\chi_g\neq0$, then such a decay rate is optimal, namely
\begin{equation}
\label{polytimottimal}
\limsup_{t\to\infty} \sqrt{t}\, \|U(t)\mathsf{C}^{-1}\| >0.
\end{equation}
\end{theorem}

The remaining of the paper is devoted to the proofs of Theorems \ref{EXP-STAB-TEO}, \ref{POL-STAB-TEO},
\ref{POL-STAB-TEO-CATTA} and \ref{POL-STAB-TEO-TIM}.


\section{Upper Resolvent Estimates for the BGP System}
\label{sceres}

\noindent
In this section, we establish some upper resolvent estimates that will be the key ingredients
in order to prove Theorem~\ref{EXP-STAB-TEO} and the first part of Theorem \ref{POL-STAB-TEO}.
To this end, for every fixed $\lambda \in \R$ and
$\widehat u=(\hat \varphi,\hat \Phi,\hat \psi,\hat \Psi,\hat w,\hat W,\hat\vartheta,\hat\eta,\hat\xi,\hat\zeta)\in \H$,
we consider the resolvent equation
$$
\i\lambda u - \mathsf{A} u = \widehat u
$$
where $u = (\varphi,\Phi,\psi,\Psi,w,W,\vartheta,\eta,\xi,\zeta) \in \D(\mathsf{A})$. In the sequel,
we denote by $c>0$ a generic constant
depending only on the structural quantities
of the problem (hence independent of $\lambda$), whose value
might change even within the same line.

The first step is to estimate
the memory variables $\eta$ and $\zeta$.
Multiplying the resolvent equation by $u$ in $\H$, taking the real part and using \eqref{dissip}, we find
$$
\frac{\varpi}{2}\big[\Gamma[\eta]  + \Gamma[\zeta] \big]= \Re \l \i\lambda u - \mathsf{A} u , u\r_\H =
\Re \l \widehat u , u\r_\H.
$$
As a consequence, we get the control
\begin{equation}
\label{DISS}
\varpi\big[\Gamma[\eta]  + \Gamma[\zeta]\big] \leq 2 \|u\|_\H \| \widehat u\|_\H.
\end{equation}
Exploiting \eqref{assnucleo1} and \eqref{assnucleo2}, the inequality above yields the bounds
\begin{align}
\label{ETA}
\varpi \|\eta\|_\M^2 &\leq c\|u\|_\H \|\widehat u\|_\H,\\\noalign{\vskip1mm}
\label{ZETA}
\varpi \|\zeta \|_\N^2 &\leq c\|u\|_\H \|\widehat u\|_\H.
\end{align}
The next step is to estimate the remaining variables of $u$. To this aim,
we write the resolvent equation componentwise
\begin{align}
\label{R1}
&\i\lambda \varphi - \Phi =\hat \varphi,\\\noalign{\vskip0.7mm}
\label{R2}
&\i\lambda \rho_1 \Phi  -k(\varphi_{x} +\psi +lw)_x - l k_0(w_{x} - l\varphi)+ l\gamma \xi =\rho_1 \hat \Phi,\\\noalign{\vskip1mm}
\label{R3}
&\i\lambda \psi - \Psi =\hat \psi,\\\noalign{\vskip0.7mm}
\label{R4}
&\i\lambda\rho_2 \Psi -b\psi_{xx} +k(\varphi_{x} +\psi +l w ) +\gamma\vartheta_{x}=\rho_2 \hat \Psi,\\\noalign{\vskip2mm}
\label{R5}
&\i\lambda w - W = \hat w,\\\noalign{\vskip1mm}
\label{R6}
&\i\lambda \rho_1 W  -k_0(w_{x} - l\varphi)_x + l k(\varphi_{x} +\psi +lw) + \gamma \xi_x= \rho_1 \hat W,\\
\label{R7}
\displaystyle
&\i\lambda \rho_3 \vartheta - \varpi\int_0^\infty \mu(s)\eta_{xx}(s)\d s + \gamma \Psi_{x} = \rho_3 \hat\vartheta,\\
\label{R8}
&\i\lambda \eta - T\eta - \vartheta = \hat\eta,\\
\label{R9}
&\i\lambda \rho_3 \xi - \varpi\int_0^\infty \nu(s)\zeta_{xx}(s)\d s + \gamma (W_{x} - l \Phi) = \rho_3 \hat\xi,\\
\label{R10}
&\i\lambda \zeta - T\zeta - \xi= \hat\zeta,
\end{align}
and we establish a number of auxiliary lemmas.

\begin{lemma}
\label{THETA}
For every $\eps \in (0,1)$ the inequality
$$
\rho_3\|\vartheta\|^2 \leq \eps \|\Psi\|^2  + \frac{c}{\eps} \|u\|_\H\|\widehat u\|_\H
$$
holds for some structural constant $c>0$ independent of $\eps$ and $\lambda$.
\end{lemma}

\begin{proof}
We introduce the space $\M_0=L^2_\mu(\R^+; L^2)$ equipped with the inner product
$$
\langle\eta_1,\eta_2\rangle_{\M_0}=\int_0^\infty \mu(s) \langle\eta_{1}(s),\eta_{2}(s)\rangle \d s.
$$
Noting that $\M \subset \M_0$ with continuous inclusion, we
multiply \eqref{R8} by $\vartheta$ in $ \M_0$ and we get
$$
g(0)\|\vartheta\|^2= \i\lambda \l\eta, \vartheta\r_{\M_0}
- \l T\eta, \vartheta\r_{\M_0}  -\l \hat \eta, \vartheta\r_{\M_0}.
$$
Exploiting \eqref{R7}, we rewrite
\begin{align*}
\i\lambda \l\eta, \vartheta\r_{\M_0}
&=\frac{\varpi}{\rho_3}\int_0^\infty \mu(s)\int_0^\infty \mu(r) \l \eta_x(s), \eta_x(r) \r \d r \d s
-\frac{\gamma}{\rho_3}\int_0^\infty \mu(s) \l \eta_x(s), \Psi \r \d s \\
& \quad - \int_0^\infty \mu(s) \l \eta(s), \hat \vartheta \r \d s.
\end{align*}
Owing to the equality above and \eqref{ETA}, we obtain the control
$$
|\i\lambda \l\eta, \vartheta\r_{\M_0} | \leq c \|\Psi\|\|\eta\|_\M + c\|u\|_\H\|\widehat u\|_\H.
$$
Moreover, integrating by parts in $s$,
we infer that
$$
\l T\eta, \vartheta\r_{\M_0} = \int_0^\infty \mu'(s)\l \eta(s), \vartheta \r \d s,
$$
where the boundary terms vanish by standard arguments (see e.g.\ \cite{Terreni}).
Hence, invoking \eqref{DISS}, we find the bound
$$
|\l T\eta, \vartheta\r_{\M_0} |
\leq c \|\vartheta\| \sqrt{\Gamma[\eta]}
\leq \frac{g(0)}{2}\|\vartheta\|^2 + c\|u\|_\H\|\widehat u\|_\H.
$$
Finally, it is apparent that
$$
|\l \hat \eta, \vartheta\r_{\M_0}|
\leq c\|u\|_\H\|\widehat u\|_\H.
$$
Collecting the estimates obtained so far and invoking \eqref{ETA}, we end up with
$$
\rho_3\|\vartheta\|^2 \leq
c \|\Psi\|\|\eta\|_\M +c\|u\|_\H\|\widehat u\|_\H \leq \eps \|\Psi\|^2  + \frac{c}{\eps} \|u\|_\H\|\widehat u\|_\H
$$
for every $\eps\in(0,1)$, as claimed.
\end{proof}

\begin{lemma}
\label{XI}
For every $\eps \in (0,1)$ the inequality
$$
\rho_3\|\xi\|^2 \leq \eps\|u\|_\H^2 + \frac{c}{\eps} \|u\|_\H\|\widehat u\|_\H
$$
holds for some structural constant $c>0$ independent of $\eps$ and $\lambda$.
\end{lemma}

\begin{proof}
The argument is similar to the previous one.
Setting $\N_0=L^2_\nu(\R^+; L^2)$ endowed with the inner product
$$
\langle\xi_1,\xi_2\rangle_{\N_0}=\int_0^\infty \nu(s) \langle\xi_{1}(s),\xi_{2}(s)\rangle \d s,
$$
and noting that $\N \subset \N_0$ with continuous inclusion, we
multiply \eqref{R10} by $\xi$ in $ \N_0$ obtaining
$$
h(0)\|\xi\|^2= \i\lambda \l\zeta, \xi\r_{\N_0}
- \l T\zeta, \xi \r_{\N_0}  -\l \hat \zeta, \xi\r_{\N_0}.
$$
Making use of \eqref{R9}, we rewrite
\begin{align*}
\i\lambda \l\zeta, \xi\r_{\N_0}
&=\frac{\varpi}{\rho_3}\int_0^\infty \nu(s)\int_0^\infty \nu(r) \l \zeta_x(s), \zeta_x(r) \r \d r \d s
-\frac{\gamma}{\rho_3}\int_0^\infty \nu(s) \l \zeta_x(s), W \r \d s \\\noalign{\vskip0.7mm}
& \quad -\frac{l\gamma}{\rho_3}\int_0^\infty \nu(s) \l \zeta(s), \Phi \r \d s
- \int_0^\infty \nu(s) \l \zeta(s), \hat \xi \r \d s.
\end{align*}
Due to the equality above and \eqref{ZETA}, we find the estimate
$$
|\i\lambda \l\zeta, \xi\r_{\N_0}|
\leq  c \|u\|_\H\|\zeta\|_\N + c\|u\|_\H\|\widehat u\|_\H.
$$
Integrating by parts in $s$ and exploiting \eqref{DISS}, we also have
$$
|\l T\zeta, \xi \r_{\N_0}| \leq c \|\xi\| \sqrt{\Gamma[\zeta]}
\leq \frac{h(0)}{2}\|\xi\|^2 + c\|u\|_\H\|\widehat u\|_\H.
$$
Finally, it is readily seen that
$$
|\l \hat \zeta, \xi\r_{\N_0}|
\leq c\|u\|_\H\|\widehat u\|_\H.
$$
Collecting these estimates and using \eqref{ZETA}, we conclude that
$$
\rho_3\|\xi\|^2 \leq
c \|u\|_\H\|\zeta\|_\N + c\|u\|_\H\|\widehat u\|_\H \leq \eps \|u\|_\H^2
+ \frac{c}{\eps} \|u\|_\H\|\widehat u\|_\H
$$
for every $\eps\in(0,1)$. The proof is finished.
\end{proof}

\begin{lemma}
\label{psi}
For every  $\eps \in (0,1)$ and every $\lambda\neq0$ the inequality
$$
b\|\psi_x\|^2 \leq \frac{\eps}{|\lambda|^2}\|u\|_\H^2  + c\eps\|\Psi\|^2 +
\frac{c}{\eps}\bigg[\frac{1}{|\lambda|}+1\bigg]\|u\|_\H\|\widehat u\|_\H
$$
holds for some structural constant $c>0$ independent of $\eps$ and $\lambda$.
\end{lemma}

\begin{proof}
We preliminary show
\begin{equation}
\label{Theta-aux}
\|\vartheta_x\| \leq c \big[1+|\lambda|\big] \sqrt{\|u\|_\H \|\widehat u\|_\H} + c\| \hat u\|_\H.
\end{equation}
To this end, multiplying \eqref{R8} by $\vartheta$ in $ \M$, we get
$$
g(0)\|\vartheta_x\|^2= \i\lambda \l\eta, \vartheta\r_{\M}
- \l T\eta, \vartheta\r_{\M}  -\l \hat \eta, \vartheta\r_{\M}.
$$
In the light of \eqref{ETA}, we estimate
$$
|\i\lambda \l\eta, \vartheta\r_{\M} | \leq c |\lambda| \|\vartheta_x\| \sqrt{\|u\|_\H \|\widehat u\|_\H},
$$
while integrating by parts in $s$ and exploiting \eqref{DISS} we infer that
$$
|\l T\eta, \vartheta\r_{\M} |
\leq c \|\vartheta_x\| \sqrt{\Gamma[\eta]}
\leq c \|\vartheta_x\| \sqrt{\|u\|_\H \|\widehat u\|_\H}.
$$
Finally, noting that
$$
|\l \hat \eta, \vartheta\r_{\M}| \leq c \|\vartheta_x\| \|\widehat u\|_\H,
$$
we arrive at \eqref{Theta-aux}.
Next, substituting \eqref{R3} into \eqref{R7}, we find the identity
$$
\i \lambda \gamma \psi_x = \varpi\int_0^\infty \mu(s)\eta_{xx}(s)\d s -\i\lambda \rho_3 \vartheta + \gamma \hat \psi_x
+\rho_3 \hat\vartheta.
$$
A multiplication by $\psi_x$ in $L^2$ entails
\begin{equation}
\label{gammo}
\i \lambda \gamma \|\psi_x\|^2 = -\varpi\int_0^\infty \mu(s)\l \eta_{x}(s), \psi_{xx}\r \d s
-\i\lambda \rho_3 \l  \vartheta , \psi_x \r +\gamma \l \hat \psi_x, \psi_x\r
+\rho_3 \l \hat\vartheta , \psi_x \r.
\end{equation}
With the aid of \eqref{R4}, we rewrite the first term in the right-hand side as
\begin{align*}
\int_0^\infty \mu(s)\l \eta_{x}(s), \psi_{xx}\r \d s  &=
\frac{\gamma}{b}\int_0^\infty \mu(s)\l \eta_{x}(s), \vartheta_{x}\r \d s
-\frac{\i\lambda\rho_2}{b}\int_0^\infty \mu(s)\l \eta_{x}(s), \Psi\r \d s\\
&\,\,\,\,+\frac{ k}{b}\int_0^\infty \mu(s)\l \eta_{x}(s), \varphi_x+\psi + lw \r \d s
-\frac{\rho_2}{b}\int_0^\infty \mu(s)\l \eta_{x}(s), \hat \Psi\r \d s.
\end{align*}
Due to the equality above, together with \eqref{ETA} and \eqref{Theta-aux}, we derive the bound
\begin{align*}
\Big|\int_0^\infty \mu(s)\l \eta_{x}(s), \psi_{xx}\r \d s \Big| &\leq
c \big[\|u\|_\H + |\lambda|\|\Psi\|\big]\|\eta\|_\M + c\|\vartheta_x\|\|\eta\|_\M
+ c\|u\|_\H\|\widehat u\|_\H \\
&\leq
c \big[\|u\|_\H + |\lambda|\|\Psi\|\big]\|\eta\|_\M
+ c\big[1+|\lambda|\big]\|u\|_\H\|\widehat u\|_\H.
\end{align*}
As a consequence, from \eqref{gammo} we infer that
$$
|\lambda| \|\psi_x\|^2
\leq c \big[\|u\|_\H + |\lambda|\|\Psi\|\big]\|\eta\|_\M
+ c |\lambda|\|\vartheta\|\|\psi_x\| +
c\big[1+|\lambda|\big]\|u\|_\H\|\widehat u\|_\H.
$$
Invoking \eqref{ETA} once more and using Lemma \ref{THETA},
we arrive at
\begin{align*}
2b\|\psi_x\|^2 &\leq \frac{c}{|\lambda|}\|u\|_\H\|\eta\|_\M + c\|\Psi\|\|\eta\|_\M +
c\|\vartheta\|\|\psi_x\| + c\bigg[\frac{1}{|\lambda|}+1\bigg]\|u\|_\H\|\widehat u\|_\H\\
&\leq\frac{\eps}{|\lambda|^2}\|u\|_\H^2  + c\eps\|\Psi\|^2 + b\|\psi_x\|^2
+\frac{c}{\eps}\bigg[\frac{1}{|\lambda|}+1\bigg]\|u\|_\H\|\widehat u\|_\H
\end{align*}
for every $\eps \in (0,1)$ and every $\lambda\neq0$. The conclusion follows.
\end{proof}

\begin{lemma}
\label{PSI}
For every $\eps>0$ small enough and every $\lambda\neq0$ the inequality
$$
\rho_2\|\Psi\|^2 \leq \frac{c\eps}{|\lambda|^2}\|u\|_\H^2
+\frac{c}{\eps^3}\bigg[\frac{1}{|\lambda|}+1\bigg]\|u\|_\H\|\widehat u\|_\H
$$
holds for some structural constant $c>0$ independent of $\eps$ and $\lambda$.
\end{lemma}

\begin{proof}
Introducing the primitive
$$P_\Psi(x)=\int_0^x {\Psi} (y) \d y \in H_0^1$$
and multiplying \eqref{R7} by $P_\Psi$ in $L^2$, we find
\begin{equation}
\label{firslemPsi}
\gamma \|\Psi\|^2 = \i\lambda \rho_3 \l \vartheta, P_\Psi \r
+ \varpi \l \eta, P_\Psi \r_\M - \rho_3 \l \hat \vartheta, P_\Psi\r.
\end{equation}
Moreover, an integration of \eqref{R4} on $(0,x)$ entails
$$
\i\lambda P_\Psi(x) = \frac{b}{\rho_2}{\psi}_x(x) - \frac{k}{\rho_2}{\varphi}(x)
-\frac{\gamma}{\rho_2} {\vartheta}(x)
- \frac{k}{\rho_2}\int_0^x [{\psi}(y) + lw(y)] \d y
 + \int_0^x  {\hat \Psi} (y) \d y.
$$
Using the identity above, together with \eqref{R1} and \eqref{R5},
we estimate the first term in the right-hand
side of \eqref{firslemPsi} as
\begin{align*}
|\i\lambda \rho_3 \l \vartheta, P_\Psi \r| &\leq c\big[\|\varphi\|+\|w\|\big]\|\vartheta\|
+ c\|\psi_x\|\|\vartheta\| + c\|\vartheta\|^2 + c\|u\|_\H\|\widehat u\|_\H\\\noalign{\vskip1.5mm}
& \leq \frac{c}{|\lambda|}\big[\|\Phi\|+\|W\|\big]\|\vartheta\|
+c\|\psi_x\|\|\vartheta\| + c\|\vartheta\|^2+c\bigg[\frac{1}{|\lambda|}+1\bigg]\|u\|_\H\|\widehat u\|_\H\\
& \leq \frac{c}{|\lambda|}\|u\|_\H\|\vartheta\|
+ c\|\psi_x\|^2 + c\|\vartheta\|^2 + c\bigg[\frac{1}{|\lambda|}+1\bigg]\|u\|_\H\|\widehat u\|_\H,
\end{align*}
for every $\lambda\neq0$.
In the light of \eqref{ETA},
the remaining terms in the right-hand
side of \eqref{firslemPsi} can be controlled as
$$
|\varpi \l \eta, P_\Psi \r_\M - \rho_3 \l \hat \vartheta, P_\Psi\r|
\leq c\|\Psi\|\|\eta\|_\M + c\|u\|_\H\|\widehat u\|_\H \leq \frac{\gamma}{2}\|\Psi\|^2 + c\|u\|_\H\|\widehat u\|_\H.
$$
Thus, invoking Lemmas \ref{THETA} and \ref{psi}, we get
\begin{align*}
2\rho_2\|\Psi\|^2 & \leq \frac{c}{|\lambda|}\|u\|_\H\|\vartheta\|
+ c\|\psi_x\|^2+ c\|\vartheta\|^2 + c\bigg[\frac{1}{|\lambda|}+1\bigg]\|u\|_\H\|\widehat u\|_\H\\
&\leq \frac{\eps}{|\lambda|^2}\|u\|_\H^2
+ c\|\psi_x\|^2 + \frac{c}{\eps}\|\vartheta\|^2
+c\bigg[\frac{1}{|\lambda|}+1\bigg]\|u\|_\H\|\widehat u\|_\H\\\noalign{\vskip0.7mm}
&\leq \frac{c\eps}{|\lambda|^2}\|u\|_\H^2 + \rho_2\|\Psi\|^2
+ \frac{c}{\eps^3}\bigg[\frac{1}{|\lambda|}+1\bigg]\|u\|_\H\|\widehat u\|_\H
\end{align*}
for every $\eps>0$ small enough and every $\lambda\neq0$ .
The thesis has been reached.
\end{proof}

\begin{lemma}
\label{w}
For every $\eps\in(0,1)$ and every $\lambda\neq0$ the inequality
$$
k_0\|w_x - l\varphi\|^2 \leq c\eps\bigg[\frac{1}{|\lambda|^2}+1\bigg]\|u\|_\H^2
+\frac{c}{\eps}\bigg[\frac{1}{|\lambda|}+1\bigg]\|u\|_\H\|\widehat u\|_\H
$$
holds for some structural constant $c>0$ independent of $\eps$ and $\lambda$.
\end{lemma}

\begin{proof}
We preliminary show
\begin{equation}
\label{xi-aux}
\|\xi_x\| \leq c \big[1+|\lambda|\big] \sqrt{\|u\|_\H \|\widehat u\|_\H} + c\| \hat u\|_\H.
\end{equation}
To this aim, multiplying \eqref{R10} by $\xi$ in $ \N$, we get
$$
h(0)\|\xi_x\|^2= \i\lambda \l\zeta, \xi\r_{\N}
- \l T\zeta, \xi\r_{\N}  -\l \hat \zeta, \xi\r_{\N}.
$$
Making use of \eqref{ZETA}, it is readily seen that
$$
|\i\lambda \l\zeta, \xi\r_{\N} | \leq c |\lambda| \|\xi_x\| \sqrt{\|u\|_\H \|\widehat u\|_\H},
$$
while integrating by parts in $s$ and owing to \eqref{DISS} we have
$$
|\l T\zeta, \xi\r_{\N}|
\leq c \|\xi_x\| \sqrt{\Gamma[\zeta]}
\leq c \|\xi_x\| \sqrt{\|u\|_\H \|\widehat u\|_\H}.
$$
Finally, we estimate
$$
|\l \hat \zeta, \xi\r_{\N}| \leq c \|\xi_x\| \|\widehat u\|_\H,
$$
and \eqref{xi-aux} follows. Next,
invoking \eqref{R1} and \eqref{R5}, we infer that
$$
W_x - l \Phi = \i\lambda(w_x -l\varphi) -\hat w_x + l \hat \varphi,
$$
and plugging such an equality into \eqref{R9} we obtain
$$
\i \lambda \gamma (w_x -l\varphi) = \varpi\int_0^\infty \nu(s)\zeta_{xx}(s)\d s -\i\lambda \rho_3 \xi
+\gamma ( \hat w_{x} - l \hat \varphi)
+ \rho_3 \hat \xi.
$$
Multiplying the identity above by $w_x -l\varphi$ in $L^2$, we are led to
\begin{align}
\label{gammol}
\i \lambda \gamma \|w_x -l\varphi\|^2 &= -\varpi\int_0^\infty \nu(s)\l \zeta_{x}(s),(w_x -l\varphi)_{x}\r \d s
-\i\lambda \rho_3 \l  \xi, w_x -l\varphi \r \\\nonumber
&\quad\, +\gamma \l\hat w_{x} - l \hat \varphi, w_x -l\varphi\r
+ \rho_3 \l \hat \xi, w_x -l\varphi\r.
\end{align}
Exploiting \eqref{R6}, we rewrite the first term in the right-hand side as
\begin{align*}
\int_0^\infty \nu(s)\l \zeta_{x}(s),(w_x -l\varphi)_{x}\r \d s  &=
\frac{\gamma}{k_0}\int_0^\infty \nu(s)\l \zeta_{x}(s), \xi_{x}\r \d s
-\frac{\i\lambda\rho_1}{k_0}\int_0^\infty \nu(s)\l \zeta_{x}(s), W\r \d s\\\noalign{\vskip0.7mm}
&\quad+\frac{ l k}{k_0}\int_0^\infty \nu(s)\l \zeta_{x}(s), \varphi_x+\psi + lw \r \d s\\\noalign{\vskip0.7mm}
&\quad-\frac{\rho_1}{k_0}\int_0^\infty \nu(s)\l \zeta_{x}(s), \hat W\r \d s.
\end{align*}
Making use of the equality above, together with \eqref{ZETA} and \eqref{xi-aux}, we derive
the control
\begin{align*}
\Big| \int_0^\infty \nu(s)\l \zeta_{x}(s),(w_x -l\varphi)_{x}\r \d s  \Big| &\leq
c\big[1+ |\lambda|\big]\|u\|_\H \|\zeta\|_\N+ c\|\xi_x\|\|\zeta\|_\N+
c\|u\|_\H\|\widehat u\|_\H\\
&\leq c\big[1+ |\lambda|\big]\|u\|_\H \|\zeta\|_\N+
c\big[1+|\lambda|\big]\|u\|_\H\|\widehat u\|_\H.
\end{align*}
As a consequence, from \eqref{gammol} we find
$$
|\lambda| \|w_x-l\varphi\|^2
\leq  c \big[1+ |\lambda|\big]\|u\|_\H \|\zeta\|_\N
+ c |\lambda|\|\xi\|\|w_x-l\varphi\| +
c \big[1+|\lambda|\big]\|u\|_\H\|\widehat u\|_\H.
$$
Appealing again to \eqref{ZETA} and using Lemma \ref{XI},
we finally get
\begin{align*}
2k_0\|w_x - l\varphi\|^2 & \leq c\bigg[\frac{1}{|\lambda|}+1\bigg]\|u\|_\H \|\zeta\|_\N
+ c\|\xi\|\|w_x-l\varphi\| +
c\bigg[\frac{1}{|\lambda|}+1\bigg]\|u\|_\H\|\widehat u\|_\H\\\noalign{\vskip0.7mm}
& \leq c\eps\bigg[\frac{1}{|\lambda|^2}+1\bigg]\|u\|_\H^2
+ k_0 \|w_x - l\varphi\|^2+
\frac{c}{\eps}\bigg[\frac{1}{|\lambda|}+1\bigg]\|u\|_\H\|\widehat u\|_\H,
\end{align*}
for every $\eps\in (0,1)$ and every $\lambda\neq0$.
The lemma has been proved.
\end{proof}

\begin{lemma}
\label{W}
For every $\eps\in(0,1)$ and every $\lambda\neq0$ the inequality
$$
\rho_1\|W\|^2 \leq c\eps\bigg[\frac{1}{|\lambda|^2}+1\bigg] \|u\|_\H^2+
\frac{c}{\eps|\lambda|^2}\| \varphi_x + \psi + lw\|^2
+ \frac{c}{\eps^3}\bigg[\frac{1}{|\lambda|}+1\bigg]\|u\|_\H\|\widehat u\|_\H
$$
holds for some structural constant $c>0$ independent of $\eps$ and $\lambda$.
\end{lemma}

\begin{proof}
Multiplying \eqref{R6} by $w$ in $L^2$ and invoking \eqref{R5}, we have
$$
\rho_1 \|W\|^2 = k_0 \l w_x-l\varphi, w_x\r + lk \l \varphi_x + \psi + lw , w\r
-\gamma \l \xi, w_x \r  - \rho_1 \l W ,\hat w\r - \rho_1 \l \hat W, w\r.
$$
Exploiting \eqref{R5} once more and appealing to Lemmas \ref{XI} and \ref{w},
the modulus of the right-hand side above is less than or equal to
\begin{align*}
& c \|u\|_\H \|w_x-l\varphi\|+
\frac{c}{|\lambda|} \|u\|_\H\|\varphi_x + \psi + lw\|
+ c\|u\|_\H\|\xi\|+ c\bigg[\frac{1}{|\lambda|}+1\bigg]\|u\|_\H\|\widehat u\|_\H\\\noalign{\vskip1mm}
&\leq \eps\|u\|_\H^2
+ \frac{c}{\eps|\lambda|^2}\|\varphi_x + \psi + lw\|^2+\frac{c}{\eps}\big[\|w_x-l\varphi\|^2 + \|\xi\|^2\big]
+c\bigg[\frac{1}{|\lambda|}+1\bigg]
\|u\|_\H\|\widehat u\|_\H\\\noalign{\vskip0.9mm}
&\leq c\eps\bigg[\frac{1}{|\lambda|^2}+1\bigg]\|u\|_\H^2
+ \frac{c}{\eps|\lambda|^2}\|\varphi_x + \psi + lw\|^2
+ \frac{c}{\eps^3}\bigg[\frac{1}{|\lambda|}+1\bigg]\|u\|_\H\|\widehat u\|_\H,
\end{align*}
for every $\eps\in (0,1)$ and every $\lambda\neq0$.
The proof is finished.
\end{proof}

We now establish two bounds on the term $\varphi_x + \psi + lw$.
The first bound will be used in the proof of Theorem  \ref{POL-STAB-TEO}, while the second one
will be used in the proof of Theorem~\ref{EXP-STAB-TEO}.

\begin{lemma}
\label{phiA}
We have the following estimates.

\smallskip
\begin{enumerate}[leftmargin=*]
\item[(i)]
For every $\eps\in(0,1)$ and every $\lambda\neq0$ the inequality
$$
k\|\varphi_x + \psi + lw\|^2 \leq c\eps\bigg[\frac{1}{|\lambda|^2}+1\bigg]\|u\|_\H^2
+\frac{c}{\eps^3}\bigg[\frac{1}{|\lambda|}+|\lambda|^2\bigg]\big[\|\Psi\|^2 + \|u\|_\H \|\hat u\|_\H\big]
$$
holds for some structural constant $c>0$ independent of $\eps$ and $\lambda$.

\medskip
\item[(ii)]
Assume that $\chi_g=0$. Then for every $\eps\in(0,1)$ and every $|\lambda|\geq1$ the inequality
$$
k\|\varphi_{x} +\psi +lw\|^2 \leq c\eps\|u\|_\H^2
+ \frac{c}{\eps^3}\big[\|\Psi\|^2 + \|u\|_\H \|\hat u\|_\H\big]
$$
holds for some structural constant $c>0$ independent of $\eps$ and $\lambda$.
\end{enumerate}

\end{lemma}

\begin{proof}
Multiplying \eqref{R4} by $\varphi_x + \psi + lw$ in $L^2$, we obtain
\begin{align}
\label{HL-1}
k\|\varphi_x + \psi + lw\|^2 &
= -\i\lambda \rho_2 \l \Psi, \varphi_x + \psi + lw\r- b \l \psi_x, (\varphi_x + \psi + lw)_x \r\\\nonumber
&\quad\,+ \gamma \l \vartheta, (\varphi_x + \psi + lw)_x\r + \rho_2\l  \hat \Psi, \varphi_x + \psi + lw\r.
\end{align}
In the light of \eqref{R1}, \eqref{R3} and \eqref{R5}, the first term in the right-hand side can be rewritten as
$$
-\i\lambda \rho_2 \l \Psi, \varphi_x + \psi + lw\r = \rho_2 \l\Psi, \Phi_x \r + \rho_2\|\Psi\|^2
+ l\rho_2\l \Psi, W\r +\rho_2 \l\Psi, \hat \varphi_{x} + \hat \psi + l \hat w\r.
$$
Appealing to \eqref{R2} and \eqref{R3}, we also write the second term in the right-hand side of \eqref{HL-1} as
\begin{align*}
- b \l \psi_x, (\varphi_x + \psi + lw)_x \r
&= -\frac{b \rho_1}{k}  \l \Psi, \Phi_x\r + \frac{blk_0}{k} \l \psi_x , w_x - l \varphi  \r \\\noalign{\vskip0.5mm}
&\quad - \frac{ bl \gamma}{k}\l \psi_x, \xi\r + \frac{b\rho_1}{k} [\l \psi_x ,\hat \Phi \r
+  \l \hat \psi_{x}, \Phi\r].
\end{align*}
Substituting the two identities above into \eqref{HL-1}, we are led to
\begin{equation}
\label{HL-2}
k\|\varphi_x + \psi + lw\|^2
= \gamma \l \vartheta, (\varphi_x + \psi + lw)_x\r
+ \Big(\rho_2 - \frac{b \rho_1}{k}\Big)  \l \Psi, \Phi_x\r + {R}_1,
\end{equation}
where
\begin{align*}
{R}_1 &= \rho_2\|\Psi\|^2 + l\rho_2\l \Psi, W\r +\frac{blk_0}{k} \l \psi_x , w_x - l \varphi  \r
- \frac{ bl \gamma}{k}\l \psi_x, \xi\r\\
& \quad +\rho_2[\l  \hat \Psi, \varphi_x + \psi + lw\r + \l\Psi, \hat \varphi_{x} + \hat \psi + l \hat w\r]
+ \frac{b\rho_1}{k} [\l \psi_x ,\hat \Phi \r
+ \l \hat \psi_{x}, \Phi\r].
\end{align*}
Exploiting now \eqref{R2} and \eqref{R7}, we find the equality
\begin{align}
\label{perno}
\gamma \l \vartheta, (\varphi_x + \psi + lw)_x\r &=
\frac{\varpi\rho_1\gamma}{\rho_3 k} \l \eta , \Phi \r_\M
-\frac{\rho_1\gamma^2}{\rho_3k}  \l \Psi, \Phi_x\r
+\frac{l\gamma^2}{k}\l \vartheta, \xi\r\\\noalign{\vskip1mm}\nonumber
&\quad - \frac{lk_0\gamma}{k} \l \vartheta , w_x - l \varphi  \r
- \frac{\gamma\rho_1}{k} [\l \vartheta , \hat \Phi \r
+ \l \hat \vartheta, \Phi\r].
\end{align}
Plugging \eqref{perno} into \eqref{HL-2} and recalling the definition of $R_1$, we readily derive the control
\begin{align*}
k\|\varphi_x + \psi + lw\|^2 \leq c\|\Phi_x\|\big[\|\Psi\| + \|\eta\|_\M\big]+
c\|u\|_\H\big[\|\Psi\|+ \|\psi_x\|  +\|\vartheta\|+\|\hat u\|_\H\big].
\end{align*}
With the aid of \eqref{R1}, we see that
$$\|\Phi_x\|\leq c|\lambda|\|\varphi_x\| + c\|\hat \varphi_x\| \leq c|\lambda|\|u\|_\H + c\|\hat u\|_\H.$$
The estimate above, together with \eqref{ETA} and Lemmas \ref{THETA} and \ref{psi}, yield the following bound
\begin{align*}
&\|\Phi_x\|\big[\|\Psi\| + \|\eta\|_\M\big]+
\|u\|_\H\big[\|\Psi\|+ \|\psi_x\|  +\|\vartheta\|+\|\hat u\|_\H\big]\\\noalign{\vskip3mm}
&\leq c|\lambda|\|u\|_\H \|\Psi\| + c|\lambda|\|u\|_\H \|\eta\|_\M +
c\|u\|_\H\big[\|\Psi\| + \|\psi_x\|  +\|\vartheta\|\big]+ c\|u\|_\H\|\hat u\|_\H\\\noalign{\vskip1mm}
&\leq \eps \|u\|_\H^2 +\frac{c}{\eps}\big[\|\Psi\|^2 + \|\psi_x\|^2
+\|\vartheta\|^2+\|u\|_\H \|\hat u\|_\H\big]
+ \frac{c|\lambda|^2}{\eps} \big[\|\Psi\|^2 + \|u\|_\H \|\hat u\|_\H \big]\\\noalign{\vskip1mm}
&\leq c\eps\bigg[\frac{1}{|\lambda|^2}+1\bigg]\|u\|_\H^2
+\frac{c}{\eps^3}\bigg[\frac{1}{|\lambda|}+|\lambda|^2\bigg]\big[\|\Psi\|^2 +\|u\|_\H \|\hat u\|_\H\big],
\end{align*}
for every $\eps\in(0,1)$ and every $\lambda\neq0$.
The proof of item (i) is finished.

We now proceed with the proof of item (ii). To this end, we multiply \eqref{R8} by $\varphi$ in~$\M$.
Invoking \eqref{R1}, we infer that
\begin{align*}
\l \eta , \Phi \r_\M &= -\l \vartheta, \varphi \r_\M - \l T \eta , \varphi \r_\M  - \l \hat \eta, \varphi \r_\M
- \l \eta, \hat \varphi \r_\M\\
& = - g(0)\l \vartheta_x, \varphi_x\r - \int_0^\infty \mu'(s)\l \eta_x(s), \varphi_x \r \d s
- \l \hat \eta, \varphi \r_\M - \l \eta, \hat \varphi \r_\M,
\end{align*}
where the second equality follows by integrating by parts in $s$ the term
$\l T \eta , \varphi \r_\M$. Thus, we get
\begin{align*}
\l \eta , \Phi \r_\M &= g(0)\l \vartheta, (\varphi_x + \psi + lw)_x \r
-g(0)\l \vartheta, \psi_x\r - l g(0) \l \vartheta, w_x\r\\
&\quad\, -\int_0^\infty \mu'(s)\l \eta_x(s), \varphi_x \r \d s
- \l \hat \eta, \varphi \r_\M - \l \eta, \hat \varphi \r_\M.
\end{align*}
Substituting the identity above into \eqref{perno}, we obtain
$$
\gamma\Big(\varpi g(0) - \frac{\rho_3 k}{\rho_1}  \Big)\l \vartheta, (\varphi_x + \psi + lw)_x \r
= \gamma^2 \l \Psi, \Phi_x \r + {R}_2,
$$
having set
\begin{align*}
{R}_2 &= \varpi \gamma\big[g(0)\l \vartheta, \psi_x\r + lg(0) \l \vartheta, w_x\r
+ \int_0^\infty \mu'(s)\l \eta_x(s), \varphi_x \r \d s\big]-\frac{l\gamma^2 \rho_3}{\rho_1}\l \vartheta, \xi\r
\\& \quad + \frac{lk_0\gamma\rho_3}{\rho_1} \l \vartheta , w_x - l \varphi  \r
+\gamma \rho_3[ \l \vartheta , \hat \Phi \r
+ \l \hat \vartheta, \Phi\r]+  \varpi \gamma[\l \hat \eta, \varphi \r_\M +  \l \eta, \hat \varphi \r_\M].
\end{align*}
At this point, introducing the number
$$\sigma_g=\varpi g(0) - \frac{\rho_3 k}{\rho_1}$$
and noting that $\chi_g = 0\, \Rightarrow\,\sigma_g\neq0$, we get
$$
\gamma \l \vartheta, (\varphi_x + \psi + lw)_x \r = \frac{\gamma^2}{\sigma_g} \l \Psi, \Phi_x \r
+ \frac{1}{\sigma_g}{R}_2.
$$
Plugging this equality into \eqref{HL-2}, we end up with
$$
k\|\varphi_x + \psi + lw\|^2
= \frac{\varpi g(0)\chi_g}{\sigma_g} \l \Psi, \Phi_x \r
+ {R}_1 + \frac{1}{\sigma_g}{R}_2.
$$
Since $\chi_g=0$ by assumption, the first term in the right-hand side vanishes. Hence, recalling
the definitions of $R_1,R_2$ and exploiting \eqref{DISS} together with Lemmas
\ref{THETA} and \ref{psi}, we estimate
\begin{align*}
k\|\varphi_x + \psi + lw\|^2 &\leq c\|u\|_\H\big[\|\Psi\|+ \|\psi_x\| +
\|\vartheta\| + \sqrt{\Gamma[\eta]} + \|\hat u\|_\H\big]\\\noalign{\vskip0.5mm}
& \leq \eps \|u\|_\H^2 + \frac{c}{\eps}\big[\|\Psi\|^2+\|\psi_x\|^2  +\|\vartheta\|^2
+ \|u\|_\H\|\hat u\|_\H\big]\\
& \leq c\eps \|u\|_\H^2 + \frac{c}{\eps^3}
\big[\|\Psi\|^2+\|u\|_\H\|\hat u\|_\H\big],
\end{align*}
for every $\eps \in (0,1)$ and every $|\lambda|\geq1$.
The proof is finished.
\end{proof}

In the proof of Theorem~\ref{EXP-STAB-TEO} a further bound on the term $\varphi_{x} +\psi +lw$ will be needed.

\begin{lemma}
\label{phiB}
Assume that $\chi_h=0$. Then for every $\eps\in(0,1)$ and every $|\lambda|\geq1$ the inequality
$$
k\|\varphi_{x} +\psi +lw\|^2 \leq c\eps \|u\|_\H^2 + \frac{c}{\eps^3}\big[\|W\|^2 + \|u\|_\H\|\hat u\|_\H\big]
$$
holds for some structural constant $c>0$ independent of $\eps$ and $\lambda$.
\end{lemma}

\begin{proof}
Multiplying \eqref{R6} by $\varphi_x + \psi + lw$ in $L^2$, we obtain
\begin{align}
\label{HG-1}
l k\|\varphi_x + \psi + lw\|^2 &
= -\i\lambda \rho_1 \l W, \varphi_x + \psi + lw\r- k_0 \l w_x-l \varphi, (\varphi_x + \psi + lw)_x \r\\\nonumber
&\quad\,+ \gamma \l \xi, (\varphi_x + \psi + lw)_x\r + \rho_1\l  \hat W, \varphi_x + \psi + lw\r.
\end{align}
Appealing to \eqref{R1}, \eqref{R3} and \eqref{R5}, we rewrite the first term in the right-hand side as
$$
-\i\lambda \rho_1 \l W, \varphi_x + \psi + lw\r = \rho_1 \l W, \Phi_x \r + \rho_1\l W, \Psi\r
+ l\rho_1\| W\|^2 +\rho_1 \l W, \hat \varphi_{x} + \hat \psi + l \hat w\r.
$$
In addition, making use of \eqref{R2} together with \eqref{R1} and \eqref{R5}, the second term in the right-hand side
of \eqref{HG-1} can be rewritten as
\begin{align*}
- k_0 \l w_x - l\varphi, (\varphi_x + \psi + lw)_x \r
&= -\frac{k_0 \rho_1}{k}  \l W, \Phi_x\r
-\frac{lk_0 \rho_1}{k}  \|\Phi\|^2 + \frac{lk_0^2}{k} \|w_x - l \varphi  \|^2 \\\noalign{\vskip0.5mm}
&\quad - \frac{l k_0 \gamma}{k}\l w_x - l\varphi, \xi\r + \frac{k_0\rho_1}{k}[\l w_x-l\varphi ,\hat \Phi \r
+ \l \hat w_{x}- l\hat \varphi, \Phi\r].
\end{align*}
Plugging the two equalities above into \eqref{HG-1}, we obtain
\begin{align}
\label{HG-2}
&lk\|\varphi_x + \psi + lw\|^2+\frac{lk_0 \rho_1}{k}  \|\Phi\|^2\\\nonumber
&= \gamma \l \xi, (\varphi_x + \psi + lw)_x\r
+ \Big(\rho_1 - \frac{k_0 \rho_1}{k}\Big)  \l W, \Phi_x\r + {R}_3,
\end{align}
having set
\begin{align*}
{R}_3 &=  \rho_1\l W, \Psi\r
+ l\rho_1\| W\|^2 + \frac{lk_0^2}{k} \|w_x - l \varphi  \|^2
- \frac{l k_0 \gamma}{k}\l w_x - l\varphi, \xi\r\\
& \quad +\rho_1[\l  \hat W, \varphi_x + \psi + lw\r+
 \l W, \hat \varphi_{x} + \hat \psi + l \hat w\r]
+ \frac{k_0\rho_1}{k}[\l w_x-l\varphi ,\hat \Phi \r
+ \l \hat w_{x}- l\hat \varphi, \Phi\r].
\end{align*}
Owing to \eqref{R2} and \eqref{R9}, we get
\begin{align}
\label{perno2}
\gamma \l \xi, (\varphi_x + \psi + lw)_x\r &=
\frac{\varpi\rho_1\gamma}{\rho_3 k} \l \zeta , \Phi \r_\N
-\frac{\rho_1\gamma^2}{\rho_3k} \l W, \Phi_x\r
-\frac{l\rho_1\gamma^2}{\rho_3k}\|\Phi\|^2
\\\noalign{\vskip1mm}\nonumber
&\quad +\frac{l\gamma^2}{k}\|\xi\|^2 - \frac{lk_0\gamma}{k} \l \xi , w_x - l \varphi  \r
- \frac{\gamma\rho_1}{k} [\l \xi , \hat \Phi \r
+ \l \hat \xi, \Phi\r].
\end{align}
Moreover, we multiply \eqref{R10} by $\varphi$ in~$\N$.
Exploiting \eqref{R1}, we are led to
\begin{align*}
\l \zeta , \Phi \r_\N &= -\l \xi, \varphi \r_\N - \l T \zeta , \varphi \r_\N  - \l \hat \zeta, \varphi \r_\N
- \l \zeta, \hat \varphi \r_\N\\
& = - h(0)\l \xi_x, \varphi_x\r - \int_0^\infty \nu'(s)\l \zeta_x(s), \varphi_x \r \d s
- \l \hat \zeta, \varphi \r_\N - \l \zeta, \hat \varphi \r_\N,
\end{align*}
where the second equality follows by integrating by parts in $s$ the term
$\l T \zeta , \varphi \r_\N$. Thus, we find
\begin{align*}
\l \zeta , \Phi \r_\N &= h(0)\l \xi, (\varphi_x + \psi + lw)_x \r
-h(0)\l \xi, \psi_x\r - l h(0) \l \xi, w_x\r\\
&\quad\, - \int_0^\infty \nu'(s)\l \zeta_x(s), \varphi_x \r \d s
- \l \hat \zeta, \varphi \r_\N - \l \zeta, \hat \varphi \r_\N.
\end{align*}
A substitution of the identity above into \eqref{perno2} yields
$$
\gamma\Big(\varpi h(0) - \frac{\rho_3 k}{\rho_1}  \Big)\l \xi, (\varphi_x + \psi + lw)_x \r
= \gamma^2 \l W, \Phi_x \r + l \gamma^2 \|\Phi\|^2 + {R}_4,
$$
where
\begin{align*}
{R}_4 &= \varpi \gamma\big[h(0)\l \xi, \psi_x\r + lh(0) \l \xi, w_x\r
+ \int_0^\infty \nu'(s)\l \zeta_x(s), \varphi_x \r \d s\big]- \frac{l\gamma^2 \rho_3}{\rho_1}\|\xi\|^2
\\& \quad + \frac{lk_0\gamma\rho_3}{\rho_1} \l \xi , w_x - l \varphi  \r
+\gamma \rho_3[ \l \xi , \hat \Phi \r
+ \l \hat \xi, \Phi\r]+  \varpi \gamma[\l \hat \zeta, \varphi \r_\N +  \l \zeta, \hat \varphi \r_\N].
\end{align*}
At this point, introducing the number
$$\sigma_h = \varpi h(0) - \frac{\rho_3 k}{\rho_1},$$
and noting that $\chi_h = 0\, \Rightarrow\,\sigma_h\neq0$, we get
$$
\gamma \l \xi, (\varphi_x + \psi + lw)_x \r = \frac{\gamma^2}{\sigma_h} \l W, \Phi_x \r
+ \frac{l \gamma^2}{\sigma_h} \|\Phi\|^2
+ \frac{1}{\sigma_h}{R}_4.
$$
Plugging the equality above into \eqref{HG-2}, we arrive at
$$
l k\|\varphi_x + \psi + lw\|^2 + \Big(\frac{lk_0 \rho_1}{k} - \frac{l \gamma^2}{\sigma_h}\Big)\|\Phi\|^2
= \frac{\varpi h(0)\chi_h}{\sigma_h} \l W, \Phi_x \r
+ {R}_3 + \frac{1}{\sigma_h}{R}_4.
$$
Being $\chi_h=0$ by assumption, the first term in the right-hand side vanishes. For the same reason, we also have
$$
\Big(\frac{lk_0 \rho_1}{k} - \frac{l \gamma^2}{\sigma_h}\Big)\|\Phi\|^2
= l\rho_1\|\Phi\|^2\geq0.
$$
Hence, recalling the definitions of $R_3,R_4$ and invoking \eqref{DISS}
together with Lemmas \ref{XI} and~\ref{w}, we estimate
\begin{align*}
k\|\varphi_x + \psi + lw\|^2 &\leq c\|u\|_\H\big[\|W\|+ \|w_x-l\varphi\| +
\|\xi\| + \sqrt{\Gamma[\zeta]} + \|\hat u\|_\H\big]\\\noalign{\vskip1.3mm}
& \leq \eps \|u\|_\H^2 + \frac{c}{\eps}\big[\|W\|^2 +\|w_x-l\varphi\|^2  +\|\xi\|^2
+ \|u\|_\H\|\hat u\|_\H\big]\\
&\leq c\eps \|u\|_\H^2 + \frac{c}{\eps^3}\big[\|W\|^2 + \|u\|_\H\|\hat u\|_\H\big],
\end{align*}
for every $\eps\in(0,1)$ and every $|\lambda|\geq1$.
The lemma is proved.
\end{proof}

\begin{lemma}
\label{PHI}
For every $\eps\in (0,1)$ and every $\lambda\neq0$ the inequality
$$
\rho_1\|\Phi\|^2 \leq c\eps \bigg[\frac{1}{|\lambda|^2}+1\bigg] \|u\|_\H^2
+ \frac{c}{\eps^3}\bigg[\frac{1}{|\lambda|}+1\bigg]
\big[\|\varphi_x + \psi + lw\|^2 + \|u\|_\H \|\hat u\|_\H\big]
$$
holds for some structural constant $c>0$ independent of $\eps$ and $\lambda$.
\end{lemma}

\begin{proof}
Multiplying \eqref{R2} by $\varphi$ in $L^2$ and exploiting \eqref{R1}, we obtain
$$
\rho_1 \|\Phi\|^2 = k\l \varphi_x + \psi + lw, \varphi_x \r -lk_0 \l w_x-l\varphi, \varphi\r
+l\gamma \l \xi, \varphi\r - \rho_1 \l \hat \Phi, \varphi\r - \rho_1 \l \Phi, \hat \varphi \r.
$$
In the light of Lemmas \ref{XI} and \ref{w},
the modulus of the right-hand side is less than or equal to
\begin{align*}
&c\|u\|_\H\big[\| \varphi_x + \psi + lw\| + \|w_x-l\varphi\| +\|\xi\| + \|\hat u\|_\H \big]\\\noalign{\vskip1.5mm}
&\leq \eps \|u\|_\H^2 + \frac{c}{\eps}\big[\|\varphi_x + \psi + lw\|^2 + \|w_x-l\varphi\|^2 +
\|\xi\|^2 +  \|u\|_\H\|\hat u\|_\H\big]\\\noalign{\vskip0.7mm}
&\leq c\eps \bigg[\frac{1}{|\lambda|^2}+1\bigg] \|u\|_\H^2
+ \frac{c}{\eps^3}\bigg[\frac{1}{|\lambda|}+1\bigg]\big[\|\varphi_x + \psi + lw\|^2 + \|u\|_\H\|\hat u\|_\H\big]
\end{align*}
for every $\eps\in (0,1)$ and every $\lambda\neq0$,
and we are finished.
\end{proof}

As a first application of the estimates obtained so far, we show that the imaginary axis $\i\R$
is contained in the resolvent set $\varrho(\mathsf{A})$ of the operator $\mathsf{A}$. Such an inclusion will
play a crucial role in the sequel.

\begin{theorem}
\label{incl}
The inclusion $\i\R \subset \varrho(\mathsf{A})$ holds.
\end{theorem}

\begin{proof}
Assume by contradiction that $\i\lambda_0\in \sigma(\mathsf{A})$ for some $\lambda_0\in\R$.
Being $\mathsf{A}$ the generator of a contraction semigroup,
$\i\lambda_0$ is necessarily an approximate eigenvalue (cf.\ \cite[Proposition~B.2]{BattyBook}), meaning that
there exists
$u_n = (\varphi_n,\Phi_n,\psi_n,\Psi_n,w_n,W_n,\vartheta_n,\eta_n,\xi_n,\zeta_n)\in\D(\mathsf{A})$
satisfying
\begin{equation}
\label{appr}
\|u_n\|_\H =1 \,\and\,
\i\lambda_0 u_n - \mathsf{A} u_n \doteq \hat u_n \to0\quad \text{in } \H.
\end{equation}
Suppose that $\lambda_0\neq0$.
Using \eqref{ETA}-\eqref{ZETA} and Lemmas~\ref{THETA}-\ref{W},
together with Lemma~\ref{phiA} item (i)
and Lemma \ref{PHI}, for every $\eps>0$ sufficiently small we get the bounds
\begin{align*}
\|u_n\|_\H^2 &\leq c\bigg[\frac{1}{|\lambda_0|^2}+|\lambda_0|^2\bigg]\Big[\eps\|u_n\|_\H^2
+ \frac{1}{\eps^3}\big[\|\varphi_{nx}+\psi_n+lw_n\|^2
+\|\Psi_n\|^2 + \|u_n\|_\H\|\hat u_n\|_\H\big]\Big]\\\noalign{\vskip0.7mm}
&\leq c\bigg[\frac{1}{|\lambda_0|^2}+|\lambda_0|^2\bigg]^2\Big[\eps\|u_n\|_\H^2
+ \frac{1}{\eps^{15}}
\big[\|\Psi_n\|^2 + \|u_n\|_\H\|\hat u_n\|_\H\big]\Big]\\\noalign{\vskip0.7mm}
&\leq c\bigg[\frac{1}{|\lambda_0|^2}+|\lambda_0|^2\bigg]^3\Big[\eps\|u_n\|_\H^2
+ \frac{1}{\eps^{63}}\|u_n\|_\H\|\hat u_n\|_\H\Big]
\end{align*}
where as before $c>0$ denotes a generic constant depending only on the structural quantities
of the problem (in particular, independent of $\eps$ and $\lambda_0$).
Fixing $\eps=\eps(\lambda_0)>0$ small enough that
$$
c\bigg[\frac{1}{|\lambda_0|^2}+|\lambda_0|^2\bigg]^3\eps <\frac12,
$$
we conclude that there exists a constant $K = K(\lambda_0)>0$ such that
$$
\|u_n\|_\H \leq K \|\hat u_n\|_\H.
$$
The latter is incompatible with \eqref{appr}. We are left to analyze the case $\lambda_0=0$.
In this situation, relation \eqref{appr} takes the form
\begin{equation}
\label{specdue}
\|u_n\|_\H =1 \,\and\,\mathsf{A} u_n \to0\quad \text{in } \H.
\end{equation}
Exploiting \eqref{dissip}, from \eqref{specdue} we obtain
\begin{equation}
\label{DISSz}
\Gamma[\eta_n]+\Gamma[\zeta_n]
= \frac{2}{\varpi}|\Re \l \mathsf{A} u_n, u_n\r_\H| \to0.
\end{equation}
Due to \eqref{assnucleo1}-\eqref{assnucleo2}, the latter implies that
$\eta_n\to0$ in $\M$ and $\zeta_n\to0$ in $\N$. Writing the second relation in \eqref{specdue}
componentwise, we also infer that $\Phi_n\to0$ in $L^2$ and $\Psi_n,W_n\to0$ in $L^2_*$. Moreover, we get
\begin{align}
\label{R2z}
& k(\varphi_{nx} +\psi_n +lw_n)_x + l k_0(w_{nx} - l\varphi_n)-l\gamma \xi_n \to 0 \quad \text{in }\, L^2,\\\noalign{\vskip0.3mm}
\label{R4z}
& b\psi_{xx} -k(\varphi_{nx} +\psi_n +l w_n) -\gamma\vartheta_{nx}\to 0\quad \text{in }\, L^2_*,\\\noalign{\vskip0.3mm}
\label{R6z}
&k_0(w_{nx} - l\varphi_n)_x - l k(\varphi_{nx} +\psi_n +lw_n) - \gamma \xi_{nx}\to 0\quad \text{in }\, L^2_*,\\\noalign{\vskip0.7mm}
\label{R8z}
&T\eta_n +\vartheta_n \to 0\quad \text{in }\, \M,\\\noalign{\vskip0.3mm}
\label{R10z}
& T\zeta_n + \xi_n \to 0 \quad \text{in }\, \N.
\end{align}
Multiplying \eqref{R8z} by $\vartheta_n$ in the space $\M_0=L^2_\mu(\R^+; L^2)$
and recalling that $\|\vartheta_n\|$ is bounded,
we find (cf.\ the proof of Lemma \ref{THETA})
$$
g(0)\|\vartheta_{n}\|^2 + \l T\eta_n, \vartheta_{n} \r_{\M_0} = g(0)\|\vartheta_{n}\|^2
+\int_0^\infty \mu'(s)\l \eta_n(s), \vartheta_n \r \d s \to0.
$$
Invoking \eqref{DISSz}, it is readily seen that
$| \int_0^\infty \mu'(s)\l \eta_n(s), \vartheta_n \r \d s|\to 0$. Thus, $\vartheta_n\to 0$
in $L^2$. Making use of \eqref{R10z} and arguing in the same way, we also have $\xi_n\to 0$ in $L^2$.
Finally, we multiply \eqref{R2z} by $\varphi_n$, \eqref{R4z} by $\psi_n$ and \eqref{R6z}
by $w_n$ in the respective spaces. Summing up, we obtain the convergence
$$
k\|\varphi_{nx} +\psi_n +lw_n\|^2 + b\|\psi_{nx}\|^2 + k_0\|w_{nx}-l\varphi_n\|^2 - \gamma \l \xi_{n},w_{nx}-l\varphi_{n}\r
-\gamma \l \vartheta_{n}, \psi_{nx}\r\to 0.
$$
Being $\|w_{nx}-l\varphi_{n}\|$, $\|\psi_{nx}\|$ bounded sequences, and recalling that $\|\vartheta_n\|,\|\xi_n\|\to0$,
the last two terms converge to zero. Hence $\varphi_{nx} +\psi_n +lw_n$ and $w_{nx}-l\varphi_n$ go to zero in $L^2$, while
$\psi_{n}$ goes to zero in $H^1_*$. Summarizing, we proved that
$\|u_n\|_\H\to0$, in contradiction with \eqref{specdue}.
\end{proof}

\begin{remark}
\label{remcompt}
The use of approximate eigenvalues (instead of eigenvalues) in the proof above
is motivated by the fact that the continuous inclusion $\D(\mathsf{A}) \subset \H$ is
not compact, due to the presence of the memory component (see e.g.\ \cite{PAZU}). This is also
the reason why the estimates given in Lemmas~\ref{THETA}-\ref{PHI}
are carried out for every $\lambda\neq0$ and not only for $|\lambda|$ large enough,
in order to perform the contradiction argument for $\lambda_0\neq0$ in the proof of Theorem \ref{incl}.
\end{remark}


\section{Proof of Theorem \ref{EXP-STAB-TEO}}
\label{GPSECTION}

\noindent
The proof is based on the celebrated Gearhart-Pr\"{u}ss theorem \cite{Ge,Pru}
(see also~\cite[Chapter 5]{BattyBook} for an historical account).

\begin{theorem}[Gearhart-Pr\"{u}ss]
\label{pruss}
A bounded semigroup $\Sigma(t)=\e^{t\mathsf{L}}$ acting on a Hilbert space is exponentially stable if and only if
$\,\i \R \subset \varrho(\mathsf{L})$ and
$$
\limsup_{|\lambda|\to\infty}\|(\i\lambda - \mathsf{L})^{-1}\| < \infty.
$$
\end{theorem}

We shall treat the two cases $\chi_g=0$ and $\chi_h=0$ separately.

\subsection*{$\boldsymbol{\chi_g=0}$}
We collect \eqref{ETA}-\eqref{ZETA} and Lemmas~\ref{THETA}-\ref{W},
together with Lemma \ref{phiA} item~(ii)
and Lemma \ref{PHI}. For every $\eps>0$ small enough and every $|\lambda|\geq 1$,
we have
\begin{align*}
\|u\|_\H^2 &\leq c \eps\|u\|_\H^2+ \frac{c}{\eps^{3}}
\big[\|\varphi_x+\psi+lw\|^2 + \|\Psi\|^2+ \|u\|_\H\|\hat u\|_\H\big]\\\noalign{\vskip0.5mm}
&\leq c\eps\|u\|_\H^2  + \frac{c}{\eps^{15}}
\big[\|\Psi\|^2+ \|u\|_\H\|\hat u\|_\H\big]\\\noalign{\vskip0.5mm}
&\leq c\eps\|u\|_\H^2  + \frac{c}{\eps^{63}} \|u\|_\H\|\hat u\|_\H,
\end{align*}
where the generic structural constant $c>0$ is independent of $\eps$ and $\lambda$.
Fixing $\eps>0$ sufficiently small that
$c\eps<1/2$, there exists a structural constant $K>0$ independent of $\lambda$ such that
$$
\|u\|_\H \leq K\|\hat u\|_\H
$$
for all $|\lambda|\geq1$.
Since Theorem \ref{incl} tells that $\i\R \subset \varrho(\mathsf{A})$, the relation above
together with the Gearhart-Pr\"{u}ss theorem yield \eqref{defexp}.
\qed

\subsection*{$\boldsymbol{\chi_h=0}$}
We appeal to \eqref{ETA}-\eqref{ZETA}, Lemmas~\ref{THETA}-\ref{W} and
Lemmas \ref{phiB}-\ref{PHI}. For every $\eps>0$ small enough and every $|\lambda|\geq 1$,
we have
\begin{align*}
\|u\|_\H^2 &\leq c \eps\|u\|_\H^2+ \frac{c}{\eps^{3}}
\big[\|\varphi_x+\psi+lw\|^2 + \|W\|^2+ \|u\|_\H\|\hat u\|_\H\big]\\\noalign{\vskip0.5mm}
&\leq c\eps\|u\|_\H^2  + \frac{c}{\eps^{15}}
\big[\|W\|^2+ \|u\|_\H\|\hat u\|_\H\big]\\\noalign{\vskip0.5mm}
&\leq c\eps\|u\|_\H^2 + \frac{c}{\eps^{31}|\lambda|^2}\|u\|_\H^2 + \frac{c}{\eps^{63}} \|u\|_\H\|\hat u\|_\H,
\end{align*}
where again the generic structural constant $c>0$ is independent of $\eps$ and $\lambda$.
Fixing $\eps>0$ sufficiently small that
$c\eps<1/2$, there is a structural constant $K>0$ independent of $\lambda$ such that
$$
\|u\|_\H^2 \leq \frac{K}{|\lambda|^2}\|u\|^2_\H + K  \|u\|_\H\|\hat u\|_\H.
$$
Therefore, for every $|\lambda|\geq \sqrt{2K}$, we end up with
$$
\|u\|_\H \leq 2K \|\hat u\|_\H.
$$
The latter estimate, the inclusion $\i\R \subset \varrho(\mathsf{A})$ and
the Gearhart-Pr\"{u}ss theorem lead to \eqref{defexp}.
\qed

\section{Proof of Theorem \ref{POL-STAB-TEO} - Decay Rate}
\label{BTSECTION}

\noindent
The argument relies on the Borichev-Tomilov theorem \cite{BT}.

\begin{theorem}[Borichev-Tomilov]
\label{BTpol}
Let $\Sigma(t)=\e^{t\mathsf{L}}$ be a bounded semigroup acting on a Hilbert space, and
assume that $\,\i \R \subset \varrho(\mathsf{L})$.
For every fixed $\alpha >0$, we have
$$
\|(\i\lambda - \mathsf{L})^{-1}\|  = {\rm O}(|\lambda|^\alpha)\quad\, \text{as }\, |\lambda|\to \infty
$$
if and only if
$$
\|\Sigma(t)\mathsf{L}^{-1}\| = {\rm O}(t^{-1/\alpha})\quad\, \text{as }\, t\to \infty.
$$
\end{theorem}

We exploit \eqref{ETA}-\eqref{ZETA} and Lemmas~\ref{THETA}-\ref{W}, together
with Lemma~\ref{phiA} item (i)
and Lemma~\ref{PHI}. For every $\eps>0$ small enough and $|\lambda|\geq 1$, we have
\begin{align*}
\|u\|_\H^2 &\leq c\eps\|u\|_\H^2 + \frac{c}{\eps^{3}}
\|\varphi_x+\psi+lw\|^2 + \frac{c|\lambda|^2}{\eps^{3}}\big[\|\Psi\|^2+ \|u\|_\H\|\hat u\|_\H\big]\\\noalign{\vskip0.7mm}
&\leq c\eps\|u\|_\H^2 + \frac{c|\lambda|^2}{\eps^{15}}
\big[\|\Psi\|^2+ \|u\|_\H\|\hat u\|_\H\big]\\\noalign{\vskip0.7mm}
&\leq c\eps\|u\|_\H^2 + \frac{c|\lambda|^2}{\eps^{63}} \|u\|_\H\|\hat u\|_\H
\end{align*}
where, as usual, $c>0$ is a generic structural
constant independent of $\eps$ and $\lambda$.
Fixing $\eps>0$ small enough that $c\eps<1/2$, there exists
a constant $K>0$ independent of $\lambda$ such that
$$
\|u\|_\H \leq K|\lambda|^2\|\hat u\|_\H
$$
for all $|\lambda|\geq1$.
In the light of Theorem \ref{incl}, such a control
and the Borichev-Tomilov theorem yield the desired conclusion \eqref{condpoly}.
\qed


\section{Lower Resolvent Estimates for the BGP system}
\label{secopti}

\noindent
In this section, we establish some lower resolvent estimates that will be the key ingredients
in order to prove the second part of Theorem \ref{POL-STAB-TEO}, namely, the optimality of the decay rate.
The following quantitative version of the Riemann-Lebesgue lemma established
in \cite{DLP} will play a crucial role.

\begin{lemma}
\label{QUANRL}
Let $f:\R^+\to\R$ be a nonincreasing, absolutely continuous and summable function.
Denote by
$$
\hat f (\lambda) = \int_0^\infty f(s)\e^{-\i \lambda s} \d s
$$
its (half) Fourier transform, and
assume that $f(s)\to f_0 \in \R$ as $s\to0$. Then
$$
\lim_{\lambda\to\infty}\lambda \hat f (\lambda) = -\i f_0.
$$
\end{lemma}

\begin{remark}
Actually, the main result of \cite{DLP} is more general and provides appropriate
asymptotic controls on $\hat f$ for a wide range of functions, possibly unbounded and nonmonotone.
The one reported above is a particular case which is enough for our purposes.
\end{remark}

\begin{lemma}
\label{stimabassoGP}
Assume that $\chi_g \, \chi_h \neq0$. Then we have (recall that the inclusion $\i\R \subset \varrho(\mathsf{A})$ has
been proved in Theorem~\ref{incl})
$$
\limsup_{\lambda \to \infty} \lambda^{-2}\|(\i \lambda - \mathsf{A})^{-1}\|>0.
$$
\end{lemma}

\begin{proof}
Calling $\omega_n = \frac{n \pi}{\ell}$ for every $n \in \mathbb{N}$,
we consider the sequence
$$
\widehat u_n(x) = \Big(0,\frac{\sin{\omega_n x}}{\rho_1},0,0,0,0,0,0,0,0\Big)\in\H.
$$
Note that
\begin{equation}
\label{normhatu}
\|\widehat u_n\|_\H=\sqrt{\frac{\ell}{2\rho_1}},\quad \forall n.
\end{equation}
Then, we study the resolvent equation
$$
\i\lambda_n u_n -\mathsf{A} u_n=\widehat u_n
$$
for some sequence of real numbers $\lambda_n\to\infty$ to be suitably chosen later.
Due to Theorem~\ref{incl}, there exists a unique solution
$$
u_n = (\varphi_n,\Phi_n,\psi_n,\Psi_n,w_n,W_n,\vartheta_n,\eta_n,\xi_n,\zeta_n)\in\D(\mathsf{A}).
$$
Using the ansatz
\begin{align*}
\varphi_n(x) &= A_n \sin{\omega_n x},\\
\psi_n(x) &= B_n \cos{\omega_n x},\\
w_n(x) &= C_n \cos{\omega_n x},\\
\vartheta_n(x) &= D_n \sin{\omega_n x},\\
\eta_n(x,s) &= d_n(s) \sin{\omega_n x},\\
\xi_n(x) &= E_n \sin{\omega_n x},\\
\zeta_n(x,s) &= e_n(s) \sin{\omega_n x},
\end{align*}
for some complex numbers $A_n,B_n,C_n,D_n,E_n$, and some complex-valued functions
$$d_n,e_n\in H^1_\mu(\R^+)\qquad \text{with}\qquad d_n(0)=e_n(0)=0,$$
after an elementary computation we get the system
\begin{equation}
\label{sys1}
\begin{cases}
p_1(n) A_n + k\omega_n B_n + l\omega_n(k+k_0)C_n + l\gamma E_n= 1,\\
\noalign{\vskip2mm}
k\omega_n A_n + p_2(n) B_n  + klC_n + \gamma \omega_n D_n =0,\\
\noalign{\vskip2mm}
l\omega_n(k +k_0)A_n + k l B_n + p_3(n)C_n + \gamma \omega_n E_n = 0,\\\noalign{\vskip1mm}
\displaystyle
\i\lambda_n\rho_3 D_n +  \varpi\, \omega_n^2 \int_0^\infty \mu(s) d_n(s) \d s - \i \lambda_n \omega_n \gamma B_n =0,\\
\noalign{\vskip0.5mm}
\i\lambda_n d_n(s) + d_n'(s) -D_n =0,\\
\noalign{\vskip0.5mm}\displaystyle
\i\lambda_n\rho_3 E_n +  \varpi\, \omega_n^2 \int_0^\infty \nu(s) e_n(s) \d s
- \i \lambda_n \gamma (\omega_n C_n + l A_n) =0,\\
\i\lambda_n e_n(s) + e_n'(s) -E_n =0,
\end{cases}
\end{equation}
having set
\begin{align*}
p_1(n) &= -\rho_1\lambda_n^2 + k\omega_n^2 + l^2 k_0,\\
p_2(n) &= -\rho_2\lambda_n^2 + b\omega_n^2 + k,\\
p_3(n) &= -\rho_1\lambda_n^2 + k_0\omega_n^2 + l^2k.
\end{align*}
Integrating the 5\textsuperscript{th} and the 7\textsuperscript{th}
equation, we find (recall that $d_n(0)=e_n(0)=0$)
$$
d_n(s) = \frac{D_n}{\i\lambda_n} (1- \e^{-\i\lambda_n s})
\,\and\,
e_n(s) = \frac{E_n}{\i\lambda_n} (1- \e^{-\i\lambda_n s}).
$$
Substituting these identities into \eqref{sys1}, we arrive at
\begin{equation}
\label{sys2}
\begin{cases}
p_1(n) A_n + k\omega_n B_n + l\omega_n(k+k_0)C_n + l\gamma E_n= 1,\\
\noalign{\vskip2mm}
k\omega_n A_n + p_2(n) B_n  + klC_n + \gamma \omega_n D_n =0,\\
\noalign{\vskip2mm}
l\omega_n(k +k_0)A_n + k l B_n + p_3(n)C_n + \gamma \omega_n E_n = 0,\\\noalign{\vskip2mm}
\lambda_n^2 \omega_n \gamma B_n + [p_4(n)- \varpi \omega_n^2\, \hat \mu (\lambda_n)] D_n =0,\\
\noalign{\vskip2mm}
\lambda_n^2 \gamma l A_n + \lambda_n^2 \omega_n \gamma C_n
+[p_5(n)- \varpi \omega_n^2\, \hat \nu (\lambda_n)] E_n=0,
\end{cases}
\end{equation}
having set
\begin{align*}
p_4(n) &= - \rho_3 \lambda_n^2 + \varpi g(0)\omega_n^2,\\
p_5(n) &= - \rho_3 \lambda_n^2 + \varpi h(0)\omega_n^2,
\end{align*}
and (cf. Lemma \ref{QUANRL})
\begin{align*}
\hat \mu(\lambda_n) = \int_0^\infty \mu(s)\e^{-\i \lambda_n s} \d s \,\and\,
\hat \nu(\lambda_n) = \int_0^\infty \nu(s)\e^{-\i \lambda_n s} \d s.
\end{align*}
For future use, we introduce the coefficient matrix $\mathbb{M}_n$ associated to system \eqref{sys2}, that is
$$
\mathbb{M}_n = \begin{bmatrix}
p_1(n) \, &\, k \omega_n \, &\, l\omega_n(k+k_0) & 0 & l\gamma \\\noalign{\vskip2.7mm}
k \omega_n \, &\, p_2(n) \, &\, kl & \gamma \omega_n & 0\\\noalign{\vskip2.7mm}
l\omega_n(k +k_0) \, &\, kl \, &\, p_3(n) & 0 & \gamma \omega_n\\\noalign{\vskip2.7mm}
0 \, &\, \lambda_n^2 \omega_n \gamma \, &\, 0 & [p_4(n)- \varpi \omega_n^2\, \hat \mu (\lambda_n)] & 0\\\noalign{\vskip2.7mm}
\lambda_n^2 \gamma l \, &\, 0 \, &\, \lambda_n^2 \omega_n \gamma & 0 & [p_5(n)- \varpi \omega_n^2\, \hat \nu (\lambda_n)]
\end{bmatrix}.
$$
We will also need the matrix $\mathbb{A}_n$ defined as
$$
\mathbb{A}_n = \begin{bmatrix}
1 \, &\, k \omega_n \, &\, l\omega_n(k+k_0) & 0 & l\gamma \\\noalign{\vskip2.7mm}
0 \, &\, p_2(n) \, &\, kl & \gamma \omega_n & 0\\\noalign{\vskip2.7mm}
0 \, &\, kl \, &\, p_3(n) & 0 & \gamma \omega_n\\\noalign{\vskip2.7mm}
0 \, &\, \lambda_n^2 \omega_n \gamma \, &\, 0 & [p_4(n)- \varpi \omega_n^2\, \hat \mu (\lambda_n)] & 0\\\noalign{\vskip2.7mm}
0 \, &\, 0 \, &\, \lambda_n^2 \omega_n \gamma & 0 & [p_5(n)- \varpi \omega_n^2\, \hat \nu (\lambda_n)]
\end{bmatrix}.
$$

\smallskip
\noindent
Once these preliminary maneuvers are done, for $c_0\in\R$ to be suitably fixed later we take
\begin{equation}
\label{deflambda}
\lambda_n = \sqrt{\frac{k\omega_n^2 + l^2 k_0 - c_0}{\rho_1}}= \sqrt{\frac{k}{\rho_1}}\omega_n
+ {\rm o}(\omega_n).
\end{equation}
With this choice, it is immediate to check that
\begin{align}
\label{P1}
p_1(n) &= c_0,\\\noalign{\vskip0.7mm}
\label{P2}
p_2(n) &= \Big(b-\frac{\rho_2 k}{\rho_1}\Big)\omega_n^2+ {\rm O}(1),\\\noalign{\vskip1.5mm}
\label{P3}
p_3(n) &= (k_0-k)\omega_n^2 + {\rm O}(1),\\\noalign{\vskip1mm}
\label{P4}
p_4(n) &= \Big(\varpi g(0) - \frac{\rho_3 k}{\rho_1}\Big)\omega_n^2+ {\rm O}(1),\\
\label{P5}
p_5(n) &= \Big(\varpi h(0) - \frac{\rho_3 k}{\rho_1}\Big)\omega_n^2+ {\rm O}(1).
\end{align}
Moreover, Lemma \ref{QUANRL}
(which is applicable due to the assumptions on $\mu$ and $\nu$) tells that
\begin{align}
\label{muhat}
\hat \mu (\lambda_n) &= -\frac{\i \mu(0)}{\lambda_n}+ {\rm o}\Big(\frac{1}{\lambda_n}\Big),\\\noalign{\vskip1.6mm}
\label{nuhat}
\hat \nu (\lambda_n) &= -\frac{\i \nu(0)}{\lambda_n}+ {\rm o}\Big(\frac{1}{\lambda_n}\Big).
\end{align}
Hence, denoting by (cf.\ the proofs of Lemmas \ref{phiA}-\ref{phiB})
$$
\sigma_g=\varpi g(0) - \frac{\rho_3 k}{\rho_1} \,\and\, \sigma_h=\varpi h(0) - \frac{\rho_3 k}{\rho_1},
$$
and exploiting \eqref{deflambda}-\eqref{nuhat}, after
a long but elementary computation we find the identity
\begin{equation}
\label{detMN}
{\rm Det}\, \mathbb{M}_n = \Big(\frac{\varpi k}{\rho_1}\Big) \alpha \omega_n^8
- \i\varpi \sqrt{\frac{\rho_1}{k}} \beta \omega_n^7
+ {\rm o}(\omega_n^7),
\end{equation}
having set
$$
\alpha = c_0 \,\chi_g \, \chi_h \Big(\frac{\varpi k}{\rho_1}\Big) g(0)h(0)
 + \chi_h  \hspace{0.3mm} \sigma_g \hspace{0.3mm} k^2 h(0)
+ \frac{\chi_g}{\rho_1}\big[\sigma_h \rho_1(k+k_0)^2-
\gamma^2 k(3k + k_0)\big]l^2 g(0)
$$
and
\begin{align*}
\beta &= c_0 \Big(\frac{\varpi k}{\rho_1}\Big) \big[ h(0)\mu(0)
\Big(b-\frac{\rho_2 k}{\rho_1}\Big)\chi_h +  g(0)\nu(0) (k_0-k)\chi_g \big]
-\Big(\frac{\varpi k^3}{\rho_1}\Big) h(0)\mu(0)\chi_h\\\noalign{\vskip0.5mm}
&\quad  -\Big(\frac{\varpi k}{\rho_1}\Big)  g(0)\nu(0) l^2 (k+k_0)^2\chi_g+ \sigma_g (k_0-k)k^2 \nu(0) \\\noalign{\vskip1mm}
&\quad + \Big(b-\frac{\rho_2 k}{\rho_1}\Big) \frac{[\sigma_h \rho_1(k+k_0)^2-
\gamma^2 k(3 k + k_0)]l^2 \mu(0)}{\rho_1}.
\end{align*}
Similarly, we also obtain the equality
\begin{equation}
\label{detAN}
{\rm Det}\, \mathbb{A}_n =  \chi_g \hspace{0.3mm} \chi_h \Big(\frac{\varpi k}{\rho_1}\Big)^2 g(0)h(0)\hspace{0.1mm}
\, \omega_n^8 +  {\rm o}(\omega_n^8).
\end{equation}
At this point, we choose $c_0$ such that $\alpha=0$, namely
$$
c_0 = - \frac{1}{\chi_g \, \chi_h }\Big(\frac{\rho_1}{\varpi k}\Big)\frac{1}{g(0)h(0)}\Big[
\chi_h  \hspace{0.3mm} \sigma_g \hspace{0.3mm} k^2 h(0)
+ \frac{\chi_g}{\rho_1}\big[\sigma_h \rho_1(k+k_0)^2-
\gamma^2 k(3k + k_0)\big]l^2 g(0)\Big]
$$
(recall that $\chi_g \hspace{0.3mm} \chi_h \neq0$ by assumption).
Substituting the expression of $c_0$ into $\beta$, after a few elementary calculations we infer that
$$
\beta = -\frac{\gamma^2  k^3}{\rho_1 \chi_g \hspace{0.3mm} \chi_h}
\bigg[\frac{\mu(0) h(0)\chi_h^2}{g(0)}+\frac{4 l^2\nu(0) g(0)\chi_g^2}{h(0)} \bigg] \doteq \beta_0 \neq 0.
$$
Therefore, \eqref{detMN} turns into
\begin{equation}
\label{detMNnew}
{\rm Det}\, \mathbb{M}_n = - \i\varpi \sqrt{\frac{\rho_1}{k}} \beta_0 \hspace{0.3mm} \omega_n^7 + {\rm o}(\omega_n^7).
\end{equation}
In particular, ${\rm Det}\, \mathbb{M}_n\neq0$
for every $n$ sufficiently large, meaning that
system~\eqref{sys2} has a unique solution $(A_n,B_n,C_n,D_n,E_n)$.
In particular, using Cramer's rule and
invoking \eqref{detAN}-\eqref{detMNnew} together with \eqref{deflambda}, we learn that
$$
|A_n| = \bigg|\frac{{\rm Det}\, \mathbb{A}_n}{{\rm Det}\, \mathbb{M}_n}\bigg|
\sim \mathfrak{c}_* \hspace{0.15mm} \lambda_n \qquad \text{where}\qquad
\mathfrak{c}_* = \Big(\frac{k}{\rho_1}\Big)^2 \varpi g(0) h(0)
\Big|\frac{\chi_g \hspace{0.3mm} \chi_h }{\beta_0} \Big|>0.
$$
Hence, appealing to \eqref{equivnorm}, we derive the controls
\begin{align*}
\|u_n\|_{\H} \geq \mathfrak{c}\|\varphi_{nx}\|
= \mathfrak{c}\sqrt{\frac{\ell}{2}}|A_n|\, \omega_n
\sim \mathfrak{c}\mathfrak{c}_* \sqrt{\frac{\ell\rho_1}{2 k}}\lambda_n^2.
\end{align*}
Due to \eqref{normhatu}, we end up with
$$
\limsup_{n\to\infty} \lambda_n^{-2} \|(\i\lambda_n - \mathsf{A})^{-1}\|
\geq \frac{\mathfrak{c}\mathfrak{c}_*\rho_1}{2\sqrt{k}}> 0.
$$
The proof is finished.
\end{proof}


\section{Proof of Theorem \ref{POL-STAB-TEO} - Optimality}
\label{proofopt}

\noindent
The proof makes use of the Batty-Duyckaerts theorem \cite{BattyDuy}
(see also \cite[Theorem 4.4.14]{BattyBook}).

\begin{theorem}[Batty-Duyckaerts]
Let $\Sigma(t)=\e^{t\mathsf{L}}$ be a bounded semigroup acting on a Banach space
with $0\in\varrho(\mathsf{L})$. Assume that $\|\Sigma(t)\mathsf{L}^{-1}\|\to 0$ as $t\to\infty$ and
let $d:[0,\infty)\to \R^+$ be a decreasing continuous function
vanishing at infinity such that $\|\Sigma(t)\mathsf{L}^{-1}\| \leq d(t)$ for all $t\geq0$.
Then $\i\R\subset\varrho(\mathsf{L})$ and there exists $C>0$
such that
$$
\|(\i\lambda-\mathsf{L})^{-1}\| \leq C d^{-1}\Big(\frac1{2|\lambda|}\Big)
$$
for all $|\lambda|$ sufficiently large.
\end{theorem}

Let $\chi_g \, \chi_h \neq0$ and assume by contradiction that \eqref{condpolyopt} is not satisfied, namely
$$\|S(t)\mathsf{A}^{-1}\| = {\rm o} (t^{-\frac{1}{2}})\quad\,\, \text{as } t \to \infty.$$
Take any decreasing continuous function $d:[0,\infty)\to \R^+$ such that
$\|S(t)\mathsf{A}^{-1}\|\leq d(t)$ and $d(t)={\rm o} (t^{-\frac{1}{2}})$.
Exploiting the Batty-Duyckaerts theorem, we learn that
$$
\|(\i\lambda-\mathsf{A})^{-1}\| \leq C d^{-1}\Big(\frac1{2|\lambda|}\Big)
$$
for some $C>0$ and every $|\lambda|$ large enough. Since
$$
d^{-1}\Big(\frac1{2|\lambda|}\Big) = {\rm o} (|\lambda|^{2})\quad\,\, \text{as } |\lambda| \to \infty,
$$
we conclude that
$$\|(\i\lambda-\mathsf{A})^{-1}\|= {\rm o} (|\lambda|^{2})\quad\,\, \text{as } |\lambda| \to \infty,
$$
in contradiction with Lemma~\ref{stimabassoGP}. \qed

\section{Proof of Theorem \ref{POL-STAB-TEO-CATTA}}
\label{profcatta}

\noindent
We shall prove the two assertions contained in Theorem \ref{POL-STAB-TEO-CATTA} (i.e.\ the
decay rate $\sqrt{t}$ and the optimality)
separately. Recall that the inclusion $\i\R \subset \varrho(\mathsf{B})$ has been proved in \cite{SARE}.

\subsection{Decay rate}\label{autopoly}
We begin by considering the semigroup $S(t)=\e^{t\mathsf{A}}:\H\to\H$ generated by equation \eqref{refabs} with the particular choice
\begin{equation}
\label{choice}
\mu(s) = \frac{1}{(\varpi\varsigma)^2}\e^{-\frac{s}{\varpi\varsigma}}
\,\and\, \nu(s)= \frac{1}{(\varpi \tau)^2}\e^{-\frac{s}{\varpi\tau}}
\end{equation}
(note that $\mu$ and $\nu$ fulfill the structural conditions stated in Subsection \ref{assmemker} with $g=g_\varsigma$
and $h=h_\tau$, where $g_\varsigma$ and $h_\tau$ are defined in \eqref{expchoice}).
Then, in the same spirit of \cite[Section 8]{TIM}, we introduce the map $\Lambda:\M\to L^2_*$ such that
$$
\Lambda \eta= - \frac{1}{\varpi \varsigma^2} \int_0^\infty \e^{-\frac{s}{\varpi\varsigma}} \eta_x(s) \d s,
$$
and the map (denoted with the same symbol)
$\Lambda:\N \to L^2_*$ such that
$$
\Lambda \zeta = - \frac{1}{\varpi \tau^2} \int_0^\infty \e^{-\frac{s}{\varpi\tau}}  \zeta_x(s) \d s.
$$
It is readily seen that
\begin{equation}
\label{boundlambda}
\varsigma\|\Lambda \eta\|^2 \leq \varpi \|\eta\|_\M^2 \,\and\,
\tau\|\Lambda \zeta\|^2 \leq \varpi \|\zeta\|_\N^2.
\end{equation}
For every given $u = (\varphi,\Phi,\psi,\Psi,w,W,\vartheta,\eta,\xi,\zeta)\in\H$, we denote by
$$
u^\Lambda = (\varphi,\Phi,\psi,\Psi,w,W,\vartheta,\Lambda\eta,\xi,\Lambda\zeta) \in \V.
$$
Let now $v_0 =(\varphi_0,\Phi_0,\psi_0,\Psi_0,w_0,W_0,\vartheta_0,p_0,\xi_0,q_0)\in \D(\mathsf{B})$ be an
arbitrarily fixed initial datum. Calling for $s>0$
$$
\eta_0(x,s)= -\frac{s}{\varpi} \int_0^x p_0(y) \d y
\and \zeta_0(x,s)= -\frac{s}{\varpi} \int_0^x q_0(y) \d y,
$$
we claim that
\begin{equation}
\label{lambdazero}
\begin{cases}
u_0 \doteq(\varphi_0,\Phi_0,\psi_0,\Psi_0,w_0,W_0,\vartheta_0,\eta_0,\xi_0,\zeta_0)\in\D(\mathsf{A}),\\
u_0^\Lambda = v_0.
\end{cases}
\end{equation}
To this end, we need to prove that
$\eta_0,\zeta_0\in \D(T)$, $\int_0^\infty \mu(s)\eta_0(s)\d s,\int_0^\infty \nu(s)\zeta_0(s)\d s \in H^2$ and
$$\Lambda \eta_0 = p_0,\,\quad\, \Lambda \zeta_0 = q_0.$$
We show the statements for $\eta_0$ (the ones for $\zeta_0$ are analogous).
Note first that $\eta_{0}\in H_0^1$ for every fixed $s$, since $p_0\in L^2_*$. Next,
recalling \eqref{choice}, we have
\begin{align*}
&\int_0^\infty \mu(s) \|\eta_{0x}(s)\|^2 \d s=
\frac{2 \varsigma}{\varpi} \|p_0\|^2 <\infty,\\\noalign{\vskip0.7mm}
&\int_0^\infty \mu(s) \|\eta'_{0x}(s)\|^2 \d s =
\frac{1}{\varpi^3 \varsigma} \|p_0\|^2 <\infty,
\end{align*}
meaning that both $\eta_0,\eta_{0}'$ belong to $\M$.
It is also apparent to see that $\|\eta_{0x}(s)\|\to 0$ as $s\to0$, yielding $\eta_0\in\D(T)$.
The condition $\int_0^\infty \mu(s)\eta_0(s)\d s \in H^2$ follows easily from the fact that $p_0\in H^1$
and $s\mu(s)\in L^1(\R^+)$.
Finally, we note that
$$
\Lambda \eta_0 = \frac{p_0}{(\varpi \varsigma)^2} \int_0^\infty s \e^{-\frac{s}{\varpi\varsigma}}\, \d s = p_0,
$$
and \eqref{lambdazero} is proved.

At this point, using
\eqref{choice} and \eqref{lambdazero},
one can readily check that $[S(t)u_0]^\Lambda$ is the unique (classical) solution to \eqref{refabsCATTA}
with initial datum $v_0$, namely
$$
T(t) v_0 = [S(t)u_0]^\Lambda \quad\,\, \forall t\geq0.
$$
As a consequence, invoking \eqref{boundlambda} and Theorem \ref{POL-STAB-TEO}, we obtain
$$
\|T(t)v_0\|_\V = \|[S(t)u_0]^\Lambda\|_\V \leq
\|S(t)u_0\|_\H \leq \frac{K}{\sqrt{t}} \|\mathsf{A} u_0\|_\H
$$
for some constant $K>0$ and every $t>0$. Exploiting again \eqref{choice}  and \eqref{lambdazero},
it is also straightforward to check that
$$
\|\mathsf{A} u_0\|_\H = \|\mathsf{B} v_0\|_\V.
$$
Therefore, we end up with
$$
\|T(t)v_0\|_\V \leq \frac{K}{\sqrt{t}} \|\mathsf{B} v_0\|_\V,
$$
for every $v_0 \in \D(\mathsf{B})$ and every $t>0$, and \eqref{polycattaott} is reached.\qed

\subsection{Optimality}\label{opticattaneo}
Our goal is to show that
\begin{equation}
\label{resol}
\limsup_{\lambda\to\infty} \lambda^{-2} \|(\i\lambda - \mathsf{B})^{-1}\| > 0,
\end{equation}
provided that $\chi_\varsigma \hspace{0.3mm} \chi_\tau\neq0$. Once this relation has been proved,
\eqref{polycattaottimal} follows exploiting the Batty-Duyckaerts theorem as in Section~\ref{proofopt}.

Analogously to the proof of Lemma \ref{stimabassoGP},
we consider the sequence
$$
\widehat v_n(x) = \Big(0,\frac{\sin{\omega_n x}}{\rho_1},0,0,0,0,0,0,0,0\Big)\in\V,
$$
where $\omega_n = \frac{n \pi}{\ell}$.
We have $\|\widehat v_n\|_\V=\sqrt{\frac{\ell}{2\rho_1}}$ for all $n\in\mathbb{N}$.
Then, we study the resolvent equation
$$
\i\lambda_n v_n -\mathsf{B} v_n=\widehat v_n
$$
for a real sequence $\lambda_n\to\infty$ to be chosen later.
The inclusion $\i\R \subset \varrho(\mathsf{B})$ tells that such an equation admits a unique solution
$$
v_n = (\varphi_n,\Phi_n,\psi_n,\Psi_n,w_n,W_n,\vartheta_n,p_n,\xi_n,q_n)\in\D(\mathsf{B}).
$$
Using the ansatz
\begin{align*}
\varphi_n(x) &= A_n \sin{\omega_n x},\\
\psi_n(x) &= B_n \cos{\omega_n x},\\
w_n(x) &= C_n \cos{\omega_n x},\\
\vartheta_n(x) &= D_n \sin{\omega_n x},\\
p_n(x) &= \tilde D_n\cos{\omega_n x},\\
\xi_n(x) &= E_n \sin{\omega_n x},\\
q_n(x) &= \tilde E_n\cos{\omega_n x},
\end{align*}
for some complex numbers $A_n,B_n,C_n,D_n,\tilde D_n,E_n,\tilde E_n$,
after an elementary calculation we get
$$
\begin{cases}
p_1(n) A_n + k\omega_n B_n + l\omega_n(k+k_0)C_n + l\gamma E_n= 1,\\\noalign{\vskip1mm}
k\omega_n A_n + p_2(n) B_n  + klC_n + \gamma \omega_n D_n =0,\\\noalign{\vskip1mm}
l\omega_n(k +k_0)A_n + k l B_n + p_3(n)C_n + \gamma \omega_n E_n = 0,\\\noalign{\vskip1mm}
\i\lambda_n\rho_3 D_n - \omega_n\tilde D_n  - \i \lambda_n \omega_n \gamma B_n =0,\\\noalign{\vskip1mm}
\i\lambda_n \varsigma \varpi \tilde D_n + \tilde D_n + \varpi\, \omega_n D_n =0,\\\noalign{\vskip1mm}
\i\lambda_n\rho_3 E_n - \omega_n\tilde E_n - \i \lambda_n \gamma (\omega_n C_n + l A_n) =0,\\\noalign{\vskip1mm}
\i\lambda_n \tau \varpi \tilde E_n + \tilde E_n + \varpi\, \omega_n E_n =0,
\end{cases}
$$
where $p_1(n),p_2(n),p_3(n)$ are defined as in the proof of Lemma \ref{stimabassoGP}.
From the 5\textsuperscript{th} and the 7\textsuperscript{th}
equation we get
$$
\tilde D_n = -\frac{\varpi\, \omega_n D_n}{\i\lambda_n \varsigma \varpi+1}
\and
\tilde E_n = -\frac{\varpi\, \omega_n E_n}{\i\lambda_n \tau \varpi+1}.
$$
Accordingly, we arrive at
$$
\begin{cases}
p_1(n) A_n + k\omega_n B_n + l\omega_n(k+k_0)C_n + l\gamma E_n= 1,\\
\noalign{\vskip4mm}
k\omega_n A_n + p_2(n) B_n  + klC_n + \gamma \omega_n D_n =0,\\
\noalign{\vskip4mm}
l\omega_n(k +k_0)A_n + k l B_n + p_3(n)C_n + \gamma \omega_n E_n = 0,\\\noalign{\vskip3mm}
\displaystyle
\lambda_n^2 \omega_n \gamma B_n + \bigg[q_1(n)
-\frac{\omega_n^2}{\varsigma(\i\lambda_n \varsigma\varpi +1)}\bigg] D_n =0,\\
\noalign{\vskip2mm}
\displaystyle
\lambda_n^2 \gamma l A_n + \lambda_n^2 \omega_n \gamma C_n
+\bigg[q_2(n)-\frac{\omega_n^2}{\tau(\i\lambda_n \tau\varpi +1)}\bigg] E_n=0,
\end{cases}
$$
where
$$
q_1(n) =  - \rho_3 \lambda_n^2 + \frac{\omega_n^2}{\varsigma}\and
q_2(n) =  - \rho_3 \lambda_n^2 + \frac{\omega_n^2}{\tau}.
$$
The system above is exactly the particular instance of~\eqref{sys2}
corresponding to the choice \eqref{choice}. Indeed, in such a case, $q_1(n)=p_4(n)$ and
$q_2(n)=p_5(n)$, while the Fourier transforms
$\hat \mu(\lambda_n)$ and $\hat \nu(\lambda_n)$ read
\begin{align*}
\hat \mu(\lambda_n) = \frac{1}{\varpi\varsigma(\i\lambda_n\varsigma\varpi+1)}\and
\hat \nu(\lambda_n) = \frac{1}{\varpi\tau(\i\lambda_n\tau\varpi+1)}.
\end{align*}
Hence, the computations in the proof of Lemma~\ref{stimabassoGP} apply.
In particular, recalling that in this situation $\chi_g=\chi_\varsigma$ and $\chi_h = \chi_\tau$, relation \eqref{resol} holds
provided that $\chi_\varsigma \hspace{0.3mm} \chi_\tau\neq0$.
\qed

\begin{remark}
As anticipated in Section \ref{coco}, making use of the maps $\Lambda$ introduced in Subsection~\ref{autopoly}
and arguing as in \cite[Section 8]{TIM}, one can prove that the semigroup $S(t)$
with the particular choice \eqref{choice} is exponentially stable if and only if the same does $T(t)$.
We refrain from doing so in this paper, leaving the details to the interested reader.
\end{remark}


\section{Proof of Theorem \ref{POL-STAB-TEO-TIM}}
\label{sezfinale}

\noindent
As in the previous section, we shall show the two assertions of Theorem \ref{POL-STAB-TEO-TIM} separately.

\subsection{Decay rate}
According to \cite{TIM}, the operator $\mathsf{C}$ satisfies the identity
\begin{equation}
\label{dissipTIM}
\Re \l \mathsf{C} z , z \r_\Z = -\frac{\varpi}{2}\Gamma[\eta] \leq 0
\end{equation}
for all $z\in\D(\mathsf{C})$,
where as before we denote by
$$\Gamma[\eta] = \int_0^\infty -\mu'(s) \|\eta_x(s)\|^2 \d s.$$
Next, for every fixed $\lambda \in \R$ and
$\widehat z\in \Z$,
we analyze the resolvent equation
$$\i\lambda z - \mathsf{C} z = \widehat z$$
where $z = (\varphi,\Phi,\psi,\Psi,\vartheta,\eta) \in \D(\mathsf{C})$.
Multiplying such an equation by $z$ in $\Z$, taking the real part
and invoking \eqref{dissipTIM}, we find
$$
\frac{\varpi}{2}\Gamma[\eta] = \Re \l \i\lambda z - \mathsf{C} z , z\r_\Z =
\Re \l \widehat z , z\r_\Z.
$$
Hence, appealing to \eqref{assnucleo1}, we get
\begin{equation}
\label{ETA-tim}
\varpi \|\eta\|_\M^2 \leq  c\|z\|_\Z \|\widehat z\|_\Z,
\end{equation}
for some  structural constant $c>0$ independent of $\lambda$.
The next lemma summarizes the needed bounds on the other variables of $z$.

\begin{lemma}
\label{LEMMATIM}
For every $\eps>0$ small enough and every $\lambda\neq0$ the inequalities
\begin{align*}
&\rho_3\|\vartheta\|^2 \leq \eps \|\Psi\|^2  + \frac{c}{\eps} \|z\|_\Z\|\widehat z\|_\Z,\\\noalign{\vskip0.7mm}
&b\|\psi_x\|^2 \leq \frac{\eps}{|\lambda|^2}\|z\|_\Z^2  + c\eps\|\Psi\|^2 +
\frac{c}{\eps}\bigg[\frac{1}{|\lambda|}+1\bigg]\|z\|_\Z\|\widehat z\|_\Z,\\\noalign{\vskip0.5mm}
&\rho_2\|\Psi\|^2 \leq \frac{c\eps}{|\lambda|^2}\|z\|_\Z^2
+\frac{c}{\eps^3}\bigg[\frac{1}{|\lambda|}+1\bigg]\|z\|_\Z \|\widehat z\|_\Z,\\\noalign{\vskip1mm}
&k\|\varphi_x + \psi\|^2 \leq \eps\|z\|_\Z^2
+\frac{c}{\eps}\big[1+|\lambda|^2\big]\big[\|\Psi\|^2
+ \|z\|_\Z \|\hat z\|_\Z\big],\\\noalign{\vskip1.7mm}
&\rho_1\|\Phi\|^2 \leq \eps \|z\|_\Z^2
+ \frac{c}{\eps}\big[\|\varphi_x + \psi\|^2 + \|z\|_\Z \|\hat z\|_\Z\big]
\end{align*}
hold for some structural constant $c>0$ independent of $\eps$ and $\lambda$.
\end{lemma}

\begin{proof}
The estimates above can be comfortably achieved revisiting the proofs
of Lemma~\ref{THETA}, Lemmas~\ref{psi}-\ref{PSI}, Lemma~\ref{phiA} item (i) and Lemma \ref{PHI}
in the limit case $l=0$.
\end{proof}

Exploiting \eqref{ETA-tim} and Lemma \ref{LEMMATIM},
for every $\eps>0$ small enough and every $\lambda\neq0$ we get
\begin{equation}
\label{tim}
\|z\|^2_\Z \leq  c\eps \bigg[\frac{1}{|\lambda|^2}+1\bigg]\|z\|_\Z^2
+\frac{c}{\eps^{15}}\bigg[\frac{1}{|\lambda|^2}+|\lambda|^2\bigg]\|z\|_\Z \|\hat z\|_\Z,
\end{equation}
where $c>0$ stands for a generic structural
constant independent of $\eps$ and $\lambda$. Using \eqref{tim}
and arguing as in the proof of Theorem~\ref{incl},
one can show that $\i\R\subset \varrho(\mathsf{C})$. The details are left to the reader.
Moreover, for every $|\lambda|\geq1$, estimate \eqref{tim} yields
$$
\|z\|^2_\Z \leq c\eps \|z\|_\Z^2 +\frac{c|\lambda|^2}{\eps^{15}}\|z\|_\Z \|\hat z\|_\Z.
$$
Hence, fixing $\eps>0$ sufficiently small that
$c\eps<1/2$, we end up with
$$
\|z\|_\Z\leq K |\lambda|^2\|\hat z\|_\Z
$$
for some constant $K>0$ independent of $\lambda$.
As in Section \ref{BTSECTION}, the control above together with the
Borichev-Tomilov theorem lead to \eqref{polytimott}. \qed

\subsection{Optimality}
We basically need to revisit the proof of Lemma \ref{stimabassoGP} in the limit case $l=0$.
To this end, setting as customary $\omega_n = \frac{n \pi}{\ell}$,
we introduce the sequence
$$
\widehat z_n(x) = \Big(0,\frac{\sin{\omega_n x}}{\rho_1},0,0,0,0\Big)\in\Z,
$$
which satisfies $\|\widehat z_n\|_\Z=\sqrt{\frac{\ell}{2\rho_1}}$ for all $n$.
Assuming $\chi_g\neq0$, we consider the resolvent equation
$$\i\lambda_n z_n -\mathsf{C} z_n=\widehat z_n$$
for a real sequence $\lambda_n\to\infty$ to be chosen later.
Since $\i\R \subset \varrho(\mathsf{C})$, there
exists a unique solution
$$
z_n = (\varphi_n,\Phi_n,\psi_n,\Psi_n,\vartheta_n,\eta_n)\in\D(\mathsf{C}).
$$
Making use of the ansatz
\begin{align*}
\varphi_n(x) &= A_n \sin{\omega_n x},\\
\psi_n(x) &= B_n \cos{\omega_n x},\\
\vartheta_n(x) &= C_n \sin{\omega_n x},\\
\eta_n(x,s) &= c_n(s) \sin{\omega_n x},
\end{align*}
for some complex numbers $A_n,B_n,C_n$ and some complex-valued function
$c_n\in H^1_\mu(\R^+)$ with $c_n(0)=0$,
after an elementary calculation we get the system
$$
\begin{cases}
r_1(n) A_n + k\omega_n B_n = 1,\\
\noalign{\vskip2.5mm}
k\omega_n A_n + r_2(n) B_n + \gamma \omega_n C_n =0,\\\noalign{\vskip1mm}
\displaystyle
\i\lambda_n\rho_3 C_n +  \varpi\, \omega_n^2 \int_0^\infty \mu(s) c_n(s) \d s - \i \lambda_n \omega_n \gamma B_n =0,\\
\noalign{\vskip1mm}
\i\lambda_n c_n(s) + c_n'(s) -C_n =0,
\end{cases}
$$
having set
\begin{align*}
r_1(n) = -\rho_1\lambda_n^2 + k\omega_n^2\and
r_2(n) = -\rho_2\lambda_n^2 + b\omega_n^2 + k.
\end{align*}
Integrating the last equation of the system above
and substituting the resulting expression into the third equation,
we arrive at
\begin{equation}
\label{sys-tim}
\begin{cases}
r_1(n) A_n + k\omega_n B_n = 1,\\
\noalign{\vskip1mm}
k\omega_n A_n + r_2(n) B_n + \gamma \omega_n C_n =0,\\\noalign{\vskip1mm}
\lambda_n^2 \omega_n \gamma B_n + [r_3(n)- \varpi \omega_n^2\, \hat \mu (\lambda_n)] C_n =0,
\end{cases}
\end{equation}
where
$$
r_3(n) = - \rho_3 \lambda_n^2 + \varpi g(0)\omega_n^2
\and
\hat \mu(\lambda_n) = \int_0^\infty \mu(s)\e^{-\i \lambda_n s} \d s.
$$
At this point, we choose
$$
\lambda_n = \sqrt{\frac{k\omega_n^2 - c_0}{\rho_1}}= \sqrt{\frac{k}{\rho_1}}\omega_n
+ {\rm o}(\omega_n),
$$
for some $c_0\in\R$ to be fixed later.
In this situation, it is immediate to check that
\begin{align*}
r_1(n) &= c_0,\\\noalign{\vskip1mm}
r_2(n) &= \Big(b-\frac{\rho_2 k}{\rho_1}\Big)\omega_n^2+ {\rm O}(1),\\\noalign{\vskip0.5mm}
r_3(n) &= \Big(\varpi g(0) - \frac{\rho_3 k}{\rho_1}\Big)\omega_n^2+ {\rm O}(1),\\
\hat \mu (\lambda_n) &= -\frac{\i \mu(0)}{\lambda_n}+ {\rm o}\Big(\frac{1}{\lambda_n}\Big),
\end{align*}
where the last equality follows from Lemma \ref{QUANRL}.
Analogously to the proof of Lemma~\ref{stimabassoGP}, we introduce
the coefficient matrix $\mathbb{M}_n$ associated to \eqref{sys-tim}, that is
$$
\mathbb{M}_n = \begin{bmatrix}
r_1(n) \, &\, k \omega_n \, & 0 \\\noalign{\vskip2.7mm}
k \omega_n \, &\, r_2(n) \, & \gamma \omega_n \\\noalign{\vskip2.7mm}
0 \, &\, \lambda_n^2 \omega_n \gamma \, & [r_3(n)- \varpi \omega_n^2\, \hat \mu (\lambda_n)]
\end{bmatrix}.
$$
Exploiting the asymptotic relations above, after an elementary computation we find
\begin{equation}
\label{detMN-tim}
{\rm Det}\, \mathbb{M}_n =\alpha \omega_n^4 - \i\varpi \mu(0) \sqrt{\frac{\rho_1}{k}} \beta \hspace{0.3mm} \omega_n^3
+ {\rm o}(\omega_n^3),
\end{equation}
having set
$$
\alpha = - c_0\, \chi_g \frac{\varpi g(0) k}{\rho_1} - k^2 \Big(\varpi g(0)-\frac{\rho_3 k}{\rho_1}\Big)
\,\,\and\,\,
\beta = k^2 - c_0 \Big(b-\frac{k \rho_2}{\rho_1}\Big).
$$
Similarly, considering the matrix
$\mathbb{A}_n$ defined as
$$
\mathbb{A}_n = \begin{bmatrix}
1 \, &\, k \omega_n \, & 0 \\\noalign{\vskip2.7mm}
0 \, &\, r_2(n) \, & \gamma \omega_n \\\noalign{\vskip2.7mm}
0 \, &\, \lambda_n^2 \omega_n \gamma \, & [r_3(n)- \varpi \omega_n^2\, \hat \mu (\lambda_n)]
\end{bmatrix},
$$
we obtain the equality
\begin{equation}
\label{detAN-tim}
{\rm Det}\, \mathbb{A}_n =- \chi_g \frac{\varpi g(0) k}{\rho_1} \omega_n^4 +  {\rm o}(\omega_n^4).
\end{equation}
Choosing $c_0$ such that $\alpha=0$, and then substituting the expression of $c_0$ into $\beta$,
we obtain (recall that $\chi_g\neq0$ by assumption)
$$
\beta = \frac{\gamma^2  k^2}{\chi_g \varpi g(0)} \doteq \beta_0 \neq 0.
$$
Accordingly, \eqref{detMN-tim} takes the form
$$
{\rm Det}\, \mathbb{M}_n = -\i\varpi \mu(0) \sqrt{\frac{\rho_1}{k}}
\beta_0 \hspace{0.5mm} \omega_n^3 + {\rm o}(\omega_n^3).
$$
In particular, ${\rm Det}\, \mathbb{M}_n\neq0$
for every $n$ sufficiently large and thus system \eqref{sys-tim} has a unique solution $(A_n,B_n,C_n)$.
Using Cramer's rule and invoking \eqref{detAN-tim}, together
with the equality above and the definition of $\lambda_n$, we infer that
$$|A_n| =\bigg|\frac{{\rm Det}\, \mathbb{A}_n}{{\rm Det}\, \mathbb{M}_n}\bigg|
 \sim \mathfrak{c}_* \hspace{0.15mm} \lambda_n\qquad
\text{where}\qquad
\mathfrak{c}_* = \frac{g(0)k}{\mu(0)\rho_1} \Big|\frac{\chi_g}{\beta_0}\Big| >0.
$$
Finally, arguing as in the last part of the proof
of Lemma \ref{stimabassoGP} (in particular, using the equivalence
between $\|\cdot\|_\Z$ and the product norm, and the fact
that $\|\widehat z_n\|_\Z$ is constant), we get
$$
\limsup_{n\to\infty} \lambda_n^{-2} \|(\i\lambda_n - \mathsf{C})^{-1}\| > 0.
$$
Relation \eqref{polytimottimal} now follows exploiting the Batty-Duyckaerts theorem as in Section~\ref{proofopt}.
\qed


%

\end{document}